\documentclass{amsart}
\usepackage[utf8]{inputenc}
\usepackage[margin=01.5 in]{geometry}
\usepackage{dirtytalk}
\usepackage[dvipsnames]{xcolor}

\usepackage{amsmath}
\usepackage{amssymb}
\usepackage{amsthm}
\usepackage{tikz-cd}
\usepackage{tikz}
\usetikzlibrary{knots}
\usetikzlibrary{arrows.meta}
\usepackage{graphicx}
\usepackage{mathtools}
\usepackage{extarrows}
\usepackage{booktabs}
\usepackage{pythonhighlight}
\usepackage{hyperref}
\usepackage[all]{xy}

\usepackage{amsmath,amsthm,amssymb,latexsym,amscd}
\usepackage{mathrsfs}

\usepackage{hyperref}
\hypersetup{pdftitle={A new condition on the Jones polynomial of a fibered positive link},pdfauthor={Lizzie Buchanan}}
\hypersetup{colorlinks=true,linkcolor=blue,anchorcolor=blue,citecolor=blue}
\numberwithin{equation}{section}
\theoremstyle{plain}
 
\newtheorem{lemma}[equation]{Lemma} 

\newtheorem{proposition}[equation]{Proposition}

\newtheorem{theorem}[equation]{Theorem}

\newtheorem{corollary}[equation]{Corollary}

\theoremstyle{remark}

\newtheorem{remark}[equation]{Remark}

\theoremstyle{definition}

\newtheorem{definition}[equation]{Definition}

\newenvironment{enumalph}
{\begin{enumerate}}
{\end{enumerate}}

\newenvironment{enumroman}
{\begin{enumerate}}
{\end{enumerate}}

\title{A new condition on the Jones polynomial of a fibered positive link}
\author{Lizzie Buchanan}
\address{Department of Mathematics, Dartmouth College, Hanover, NH 03755}
\email{elizabeth.m.buchanan.gr@dartmouth.edu}
\urladdr{\url{http://math.dartmouth.edu/~ebuchanan/}}

\begin{document}

\maketitle

\begin{abstract}
We give a new upper bound on the maximum degree of the Jones polynomial of a  fibered positive link. In particular, we prove that the maximum degree of the Jones polynomial of a fibered positive knot is at most four times the minimum degree.
Using this result, we can complete the classification of all knots of crossing number $\leq 12$ as positive or not positive, by showing that the seven remaining knots for which positivity was unknown are not positive. That classification was also done independently at around the same time by Stoimenow.
\end{abstract}

\section{Introduction}

\subsection{Motivation}

We would ultimately like to produce an infinite family of knots that have an almost-alternating minimal diagram. Almost-alternating diagrams and almost-alternating knots were first introduced by  Adams et al. \cite{Ad+}. An example of an almost-alternating knot that has a minimal (with respect to crossing number) diagram that is itself almost-alternating appeared in a 2009 paper of Diao, Ernst, and Stasiak \cite{Di+}. Can we construct an infinite family of knots with this property?

Almost-positive-alternating knot diagrams (diagrams for which a single crossing change would produce a positive alternating diagram) seem like a good candidate of sub-families to work with. In many examples, we come close to proving that an almost-positive-alternating knot diagram with crossing number $c$ is minimal, except we have two cases we need to rule out: Case 1, that there could be an almost-positive diagram with $c-1$ crossings, and Case 2, that there could be a positive diagram with $c-1$ crossings. 

Many of our constructed examples of almost-positive-alternating knot diagrams have the additional property that $V_1=0$ (the second coefficient of the Jones polynomial is $0$). The work below
can help us rule out Case 2, by providing a bound on the maximum degree of the Jones polynomial of a positive knot in the case that $V_1=0$, which we know from Stoimenow \cite{Stoimenow} is exactly the case where the positive knot is a fibered knot.

\subsection{Main Results}

The main result of this paper appears as:\\

\textbf{Corollary \ref{max deg V leq 4 min deg V }:} \textit{If $K$ is a fibered positive knot (with Jones polynomial $V$), then $$\max \deg V \leq 4\min\deg V.$$}

On the KnotInfo website \cite{KnotInfo}, there are only seven knots listed for which positivity is unknown, and determining the positivity (up to mirroring) of these knots is presented as an open question in \cite{AT17}. With our result above, we can resolve this question and show that the knots $12_{n148}, 12_{n276}, 12_{n329}, 12_{n366}, 12_{n402}, 12_{n528},$ and $12_{n660}$ are not positive. This completes the positivity classification of all knots with crossing number $\leq 12$. 

\begin{remark}
Updates for the reader: Since posting the first version of this paper, our motivating question has been answered in the affirmative by Stoimenow, who in \cite{St22} gives a construction of an infinite family of knots that have a minimal almost-alternating diagram. 

We have also learned that Stoimenow independently completed the positivity classification of knots up to $12$ crossings at around the same time as our work for this paper. In fact, Stoimenow has compiled lists of all non-alternating prime knots of crossing number $\leq 15$. These lists can be found at \cite{StW}. 
\end{remark}

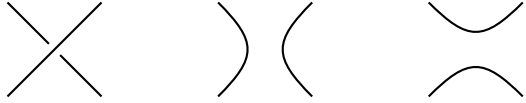
\begin{figure}[h]
    \centering
    \begin{tikzpicture}[every path/.style={thick}, every
node/.style={transform shape, knot crossing, inner sep=2.75pt}, scale=0.7]
    \node (tl) at (-1, 1) {};
    \node (tr) at (1, 1) {};
    \node (bl) at (-1, -1) {};
    \node (br) at (1, -1) {};
    \node (c) at (0,0) {};

    \draw (bl) -- (tr);
    \draw (br) -- (c);
    \draw (c) -- (tl);
    
    \begin{scope}[xshift=4cm]
    \node (tl) at (-1, 1) {};
    \node (tr) at (1, 1) {};
    \node (bl) at (-1, -1) {};
    \node (br) at (1, -1) {};

    \draw (bl) .. controls (bl.8 north east) and (tl.8 south east) .. (tl);
    \draw (br) .. controls (br.8 north west) and (tr.8 south west) .. (tr);
    \end{scope}

    \begin{scope}[xshift=8cm]
    \node (tl) at (-1, 1) {};
    \node (tr) at (1, 1) {};
    \node (bl) at (-1, -1) {};
    \node (br) at (1, -1) {};

    \draw (bl) .. controls (bl.8 north east) and (br.8 north west) .. (br);
    \draw (tl) .. controls (tl.8 south east) and (tr.8 south west) .. (tr);
    \end{scope}

    \end{tikzpicture}
    \caption{ A crossing (left), its $A$-smoothing (middle), and $B$-smoothing (right)}
    \label{A-smooth and B-smooth}
\end{figure}

\subsection*{Acknowledgements}
The author would like to thank her graduate research advisor, Vladimir Chernov. Additional thanks to Ina Petkova and C.-M. Michael Wong for comments on a draft of this paper. 

\section{Bound on max degree of the Jones polynomial of a positive knot}

Throughout the following, for a link diagram $D$ we let: 
\begin{itemize}
    \item $c(D)$ be the crossing number 
    \item $w(D)$ be the writhe
    \item $n(D)$ be the number of link components
    \item $s(D)$ be the number of Seifert circles
    \item $g(D)$ be the Seifert genus, which for a non-split link is found by $g(D)=\frac{c(D)-s(D)+2 - n(D)}{2}$,
    \item $|A(D)|$ be the number of $A$-circles (sometimes we may just use $|A|$ if there is only one diagram in question)
    \item $|B(D)|$ (or just $|B|$) be the number of $B$-circles
    \item $V$ be the Jones polynomial
    \item $V_i$ be the $i$th coefficient of the Jones polynomial. \\
    This means that if $d$ is the lowest degree of a non-zero term in $V$, we write $V_0$ to refer to the coefficient of the term $t^d$, and then $V_i$ is the (not necessarily non-zero) coefficient of the term $t^{d+i}$. 
\end{itemize}

Whenever we refer to \say{an arc} of a diagram, we mean a portion of a strand that goes between two crossings, so an arc ends when it reaches any crossing, not just an undercrossing.

\begin{definition}
Let $D$ be a link diagram. Smooth every crossing in the $A$-direction (see Figure \ref{A-smooth and B-smooth}). Create the dual graph, where each $A$-circle corresponds to a vertex in the dual graph, and each crossing shared between $A$-circles corresponds to an edge between vertices in the graph. This graph is called the  \textbf{$A$-state graph}. (Figure \ref{reduced A-state graph example})
\end{definition}

\begin{definition}
Let $D$ be a link diagram. The \textbf{reduced $A$-state graph} is the result of removing duplicate edges from the $A$-state graph, so that vertices in the reduced $A$-state graph share a single edge if and only if their corresponding $A$-circles share at least one crossing in the knot diagram. 
\end{definition}

\begin{figure}[h]
    \centering
    \begin{tikzpicture}[every path/.style={thick}, every
node/.style={transform shape, knot crossing, inner sep=2pt}]
    \begin{scope}[xshift=0cm, scale=0.8]
    \node (tlo) at (-6.25,2.25) {};
    \node (tro) at (-3.75,2.25) {};
    \node (so) at (-5,1.25) {};
    \node (mlo) at (-5.5,0) {};
    \node (mro) at (-4.5, 0) {};
    \node (bo) at (-5,-1.25) {};
    \draw (bo.center) .. controls (bo.4 north west) and (mlo.4 south west) ..
    (mlo.center);
    \draw  (bo) .. controls (bo.4 north east) and (mro.4 south east).. (mro);
    \draw  [-{Stealth}] (mlo) .. controls (mlo.8 north west) and (so.3 south west) ..
    (so);
    \draw [-{Stealth}] (mlo.center) .. controls (mlo.8 north east) and (mro.2 north
    west) .. (mro);
    \draw  (mro.center) .. controls (mro.4 north east) and (so.8 south east) ..
    (so.center);
    \draw (mro.center) .. controls (mro.8 south west) and (mlo.3 south east) ..
    (mlo);
    \draw (so) .. controls (so.4 north east) and (tro.8 south west) .. (tro);
    \draw  (so.center) .. controls (so.8 north west) and (tlo.8 south
    east) .. (tlo.center);
    \draw (tlo.center) .. controls (tlo.16 north west) and (tro.16 north
    east) .. (tro);
    \draw (bo.center) .. controls (bo.16 south east) and (tro.16 south east) ..
    (tro.center);
    \draw (bo) .. controls (bo.16 south west) and (tlo.16 south
    west) .. (tlo);
    \draw (tlo) .. controls (tlo.4 north east) and (tro.4 north
    west) .. (tro.center);
    \end{scope}
    
    \begin{scope}[xshift=-0.75cm, scale=0.8]
    \node (tl) at (-1.25,2.25) {};
    \node (tltl) at (-1.3,2.3) {};
    \node (tlbl) at (-1.3,2.2) {};
    \node (tltr) at (-1.2,2.3) {};
    \node (tlbr) at (-1.2,2.2) {};
    \node (tr) at (1.25,2.25) {};
    \node (trtl) at (1.2,2.3) {};
    \node (trbl) at (1.2,2.2) {};
    \node (trtr) at (1.3,2.3) {};
    \node (trbr) at (1.3,2.2) {};
    \node (stl) at (-0.05,1.3) {};
    \node (str) at (0.05,1.3) {};
    \node (sbl) at (-0.05,1.2) {};
    \node (sbr) at (0.05,1.2) {};
    \node (ml) at (-0.5,0) {};
    \node (mltl) at (-0.55,0.05) {};
    \node (mlbl) at (-0.55,-0.05) {};
    \node (mltr) at (-0.45,0.05) {};
    \node (mlbr) at (-0.45,-0.05) {};
    \node (mr) at (0.5, 0) {};
    \node (mrtl) at (0.45, 0.05) {};
    \node (mrbl) at (0.45, -0.05) {};
    \node (mrtr) at (0.55, 0.05) {};
    \node (mrbr) at (0.55, -0.05) {};
    \node (b) at (0,-1.25) {};
    \node (btl) at (-0.05, -1.2) {};
    \node (bbl) at (-0.05, -1.3) {};
    \node (btr) at (0.05, -1.2) {};
    \node (bbr) at (0.05, -1.3) {};
    
    \draw [gray] (btl.center) .. controls (btl.4 north west) and (mlbl.4 south west) ..
    (mlbl.center);
    \draw [gray] (btr.center) .. controls (btr.4 north east) and (mrbr.4 south east)
    .. (mrbr.center);
    \draw [gray] (mltl.center) .. controls (mltl.8 north west) and (sbl.3 south west) ..
    (sbl.center);
    \draw [green] (mltr.center) .. controls (mltr.4 north east) and (mrtl.4 north
    west) .. (mrtl.center);
    \draw [gray] (mrtr.center) .. controls (mrtr.4 north east) and (sbr.8 south east) ..
    (sbr.center);
    \draw [green] (mrbl.center) .. controls (mrbl.4 south west) and (mlbr.4 south east) ..
    (mlbr.center);
    \draw [blue] (str.center) .. controls (str.8 north east) and (trbl.8 south west) .. (trbl.center);
    \draw [blue] (stl.center) .. controls (stl.8 north west) and (tlbr.8 south
    east) .. (tlbr.center);
    \draw [magenta] (tltl.center) .. controls (tltl.16 north west) and (trtr.16 north
    east) .. (trtr.center);
    \draw [blue] (bbr.center) .. controls (bbr.16 south east) and (trbr.16 south east) ..
    (trbr.center);
    \draw [blue] (bbl.center) .. controls (bbl.16 south west) and (tlbl.16 south
    west) .. (tlbl.center);
    \draw [magenta] (tltr.center) .. controls (tltr.4 north east) and (trtl.4 north
    west) .. (trtl.center);
    
    \draw [blue] (stl.center) -- (str.center) {};
    \draw [gray] (sbl.center) -- (sbr.center) {};
    \draw [gray] (mltl.center) -- (mlbl.center) {};
    \draw [green] (mltr.center) -- (mlbr.center) {};
    \draw [green] (mrtl.center) -- (mrbl.center) {};
    \draw [gray] (mrtr.center) -- (mrbr.center) {};
    \draw [magenta] (tltl.center) -- (tltr.center) {};
    \draw [blue] (tlbl.center) -- (tlbr.center) {};
    \draw [magenta] (trtl.center) -- (trtr.center) {};
    \draw [blue] (trbl.center) -- (trbr.center) {};
    \draw [gray] (btl.center) -- (btr.center) {};
    \draw [blue] (bbl.center) -- (bbr.center) {};
    \end{scope}

    
    \begin{scope}[xshift=2cm, scale =0.8]
    \node (magenta) at (0, 2.75) {};
    \node (m) at (0, 3) {};
    \node (blue) at (0, 1) {};
    \node (c) at (0, .75) {};
    \node (cr) at (0.25, 0.75) {};
    \node (orange) at (1.5, 0.75) {};
    \node (o) at (1.75, 0.75) {};
    \node (ou) at (1.75, 0.5) {};
    \node (ol) at (1.5, 0.75) {};
    \node (green) at (1.75, -1) {};
    \node (g) at (1.75, -1.25) {};

    \draw [magenta] (m) circle[radius=0.32cm];
    \draw [blue] (c) circle[radius=0.32cm];
    \draw [gray] (o) circle[radius=0.32cm];
    \draw [green] (g) circle[radius=0.32cm];
    \draw (blue) .. controls (blue.3 north east) and (magenta.3 south east) .. (magenta);
    \draw (blue) .. controls (blue.3 north west) and (magenta.3 south west) .. (magenta);
    \draw (orange) .. controls (orange.3 north west) and (cr.3 north east) .. (cr);
    \draw (orange) .. controls (orange.3 south west) and (cr.3 south east) .. (cr);

    \draw (green) .. controls (green.3 north west) and (ou.3 south west) .. (ou);
    \draw (green) .. controls (green.3 north east) and (ou.3 south east) .. (ou);
    \end{scope}
    
    \begin{scope}[xshift=4.65cm, scale=0.8]
    \node (magenta) at (0, 2.75) {};
    \node (m) at (0, 3) {};
    \node (cyan) at (0, 1) {};
    \node (c) at (0, .75) {};
    \node (cr) at (0.25, 0.75) {};
    \node (orange) at (1.5, 0.75) {};
    \node (o) at (1.75, 0.75) {};
    \node (ou) at (1.75, 0.5) {};
    \node (green) at (1.75, -1) {};
    \node (g) at (1.75, -1.25) {};
    
    \draw [magenta] (m) circle[radius=0.32cm];
    \draw [blue] (c) circle[radius=0.32cm];
    \draw [gray] (o) circle[radius=0.32cm];
    \draw [green] (g) circle[radius=0.32cm];
    \draw (cyan) -- (magenta);
    \draw (orange) -- (cr);
    \draw (green) -- (ou);
    \end{scope}
\end{tikzpicture} 
\caption{ (Left to right:) A positive link diagram, its $A$-circles, its $A$-state graph, and its reduced $A$-state graph }
\label{reduced A-state graph example}
\end{figure}
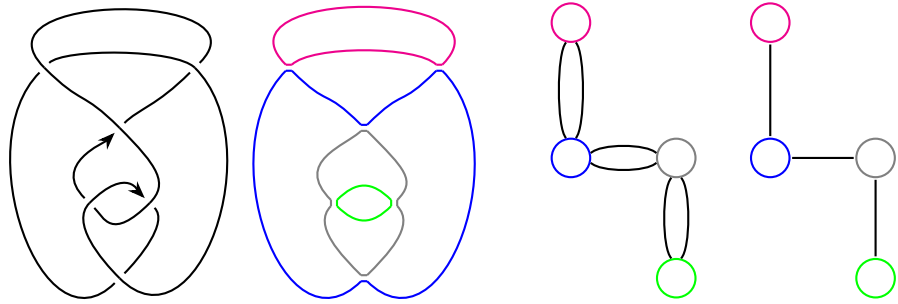

\begin{definition}
A(n oriented) link diagram is called \textbf{positive} if every crossing in that diagram is positive. (See Figure \ref{positive and negative crossing}.) A link is a \textbf{positive link} if it has a positive diagram.  
\end{definition}

\begin{figure}[h]
\centering
    \begin{tikzpicture}[knot gap=9, scale=0.5]
        
        \begin{scope}[xshift=5cm]
        \draw[thick, -{Stealth}] (0.5,-0.5) -- (-1,1);
        \draw[thick, -{Stealth}] (-1,-1) -- (1,1);
        \draw[thick, knot] (1,-1) -- (-0.5,0.5);
        \draw[thick, knot] (-1,-1) -- (0.5,0.5);
        \end{scope}
        
        \begin{scope}[xshift=9cm]
        \draw[thick, -{Stealth}] (0.5,-0.5) -- (-1,1);
        \draw[thick, -{Stealth}] (-1,-1) -- (1,1);
        \draw[thick, knot] (-1,-1) -- (0.5,0.5);
        \draw[thick, knot] (1,-1) -- (-0.5,0.5);
        \end{scope}
        
    \end{tikzpicture} 
    \caption{ A positive crossing (left) and a negative crossing (right)} \label{positive and negative crossing}
\end{figure}
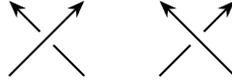

In everything that follows, we always assume we are working with non-split links.

\begin{remark}
In a positive diagram, $A$-smoothings are equivalent to smoothing according to Seifert's algorithm. So in a positive diagram, $|A(D)|=s(D)$. 
\end{remark}

\begin{theorem}\label{V_1 tree}(Stoimenow)\cite{Stoimenow} 

Let $L$ be a positive link with positive diagram $D$. Then coefficient $V_1$ of the Jones polynomial satisfies:
$$(-1)^{n(L)-1}V_1=s(D)-1-\#(\text{pairs of Seifert circles that share at least one crossing}).$$ 

This means that $V_1=0$ exactly when the reduced $A$-state graph is a tree.
\end{theorem}

So if $V_1=0$ for a positive diagram, then there are exactly $s(D)-1$ pairs of Seifert circles in the diagram that share at least one crossing. (See Figure \ref{reduced A-state graph example}.) For more details on $A$-state graphs, the Jones polynomial and its coefficients, and properties of positive knots (and almost-positive knots), see \cite{Stoimenow} and \cite{St05}. 

\subsection{Positive link diagram $D$ with $V_1=0$ must have $4\min \deg V \geq c(D)$} We begin with theorems from knot theory folklore:  
\begin{theorem} \label{min jones degree = genus of pos}
If $L$ is a positive link of $n$ components with positive diagram $D$, then $$\min \deg V_L = \frac{c(D)-s(D)+1}{2}.$$
\end{theorem}

\begin{lemma}\label{c(D) leq 4g}
Let $D$ be a reduced positive link diagram (i.e. with no nugatory crossings) such that $V_1=0$. 
Then $4\min \deg V \geq c(D)$.
\end{lemma}

\begin{proof}

Since $D$ is positive and $V_1=0$, by Theorem \ref{V_1 tree} we know that the reduced $A$-state graph is a tree. 

Consider a pair of adjacent vertices in the tree. If the corresponding vertices were connected by only one edge in the (unreduced) $A$-state graph, then the diagram would not be reduced. Since our diagram $D$ is reduced, we conclude that any pair of adjacent vertices in the reduced $A$-state graph are connected by at least two edges in the $A$-state graph. Therefore the number of crossings in $D$ is at least $2(s-1)$. 

Observe that by Theorem \ref{min jones degree = genus of pos},
\begin{align*}
    \frac{c-s+1}{2} &= \min \deg V\\
    2c-2(s-1) &= 4\min \deg V\\
    c &= 2(s-1)+4\min \deg V-c
\end{align*}

As $c\geq 2(s-1)$, we must have that $4\min \deg V-c\geq 0 $ and so $4\min \deg V \geq c$.

\end{proof}

\begin{corollary}
Let $K$ be a positive link with $V_1=0$. Then $K$ has a positive diagram $D$ for which $4\min \deg V \geq  c(D)$. 
\end{corollary}


\subsection{Balanced and Burdened link diagrams }

\begin{proposition} \label{|n(D_i) - n(D_0)| leq i}
Let $D_0$ be a link diagram. 
\begin{enumroman}
    \item Let $D_1$ be a diagram obtained by smoothing a crossing according to Seifert's algorithm. Then
    \begin{enumalph}
        \item $|n(D_1)-n(D_0)|=1$, and 
         \item 
         $||B(D_1)| - |B(D_0)|| = \begin{cases} 0 &\text{if a negative crossing is smoothed, and }\\
         1 &\text{if a positive crossing is smoothed.}
         \end{cases} $ 
    \end{enumalph}

    \item Let $D_0$ be a link diagram, and $(x_1, \dots, x_m)$ a sequence of distinct crossings in $D_0$. Let $D_i$ be the diagram obtained by smoothing (according to Seifert's algorithm) crossings $x_1$ through $x_i$. Then 
    
    \begin{enumalph}
        \item $|n(D_i) - n(D_0)| \leq i$, and 
         \item 
        $||B(D_i)| - |B(D_0)|| \leq i$.
    \end{enumalph}

\end{enumroman}
\end{proposition}

\begin{proof}
Part $(i)(a)$ is clear: smoothing a crossing involving one component creates a new component, and smoothing a crossing involving two components combines them into one. For $(i)(b)$ - $B$-smoothing a negative crossing does not change the total number of $B$-circles involved, but $B$-smoothing a positive crossing will either increase or decrease the number of $B$-circles by $1$. \\

For part $(ii)$, we rewrite $|n(D_i)-n(D_0)|$ as a telescoping sum to obtain the inequality \begin{align*}
    |n(D_i)-n(D_0)| &=|n(D_i)-n(D_{i-1})+n(D_{i-1}) - n(D_{i-2}) + \dots + n(D_{1}) - n(D_0)| \\
    &\leq \sum_{k=0}^{i-1} |n(D_{k+1})-n(D_{k})|, \hspace{5pt} \text{ which by Part $(a)$ is}\\
    &= \sum_{k=0}^{i-1} 1 \\
    &= i.
\end{align*}
Parallel argument for $|B|$. 
\end{proof}

We now introduce the concepts of a Balanced link diagram and a Burdened link diagram. 

\begin{definition}
A \textbf{Balanced link diagram} is a (non-split) positive link diagram in which every pair of $A$-circles share exactly $0$ or $2$ crossings, and the reduced $A$-state graph is a tree. 
\end{definition}

An example appears in Figure \ref{reduced A-state graph example}. Observe that if $s(D)$ is the number of Seifert circles, then in a Balanced link diagram $D$ we have that $c(D)=2(|A(D)|-1)=2(s(D)-1)$. 

\vspace{6pt}

\begin{theorem}\label{comp=Bcircles}
Let $D$ be a Balanced link diagram. Then $n(D)=|B(D)|$. 
(The number of components of $D$ is equal to the number of $B$-circles.)
\end{theorem}

This is the key theorem needed to prove our main result, and the last section of this paper leads up to a proof of Theorem \ref{comp=Bcircles}.\\


\begin{definition}
A \textbf{Burdened link diagram} is a (non-split) positive link diagram in which every pair of $A$-circles share $0$ or \underline{at least $2$} crossings, and the reduced $A$-state graph is a tree.
\end{definition}

Observe that for all Burdened link diagrams, there exists a (not necessarily unique, possibly empty) sequence of crossings $(x_1, \dots, x_m)$ such that after smoothing every $x_i$, we are left with a Balanced link diagram. We call such a sequence a \textbf{Balancing sequence}.\\ 

\begin{lemma} \label{|B(D_0)| leq 2m + n(D_0)}
Let $D_0$ be a Burdened link diagram, and $(x_1, \dots, x_m)$ a Balancing sequence for $D_0$. Then $$|B(D_0)| \leq 2m + n(D_0).$$
\end{lemma}

\begin{proof}
Proposition \ref{|n(D_i) - n(D_0)| leq i} tells us that $n(D_m) \leq m + n(D_0)$ and $|B(D_0)|\leq m + |B(D_m)|$. By definition of Balancing sequences, diagram $D_m$ is Balanced, and so $n(D_m)=|B(D_m)|$ by Theorem \ref{comp=Bcircles}. Thus 
\begin{equation*}
    |B(D_0)|\leq m + |B(D_m)| = m + n(D_m) \leq  m + (m + n(D_0)) = 2m + n(D_0).
\end{equation*}
\end{proof}

\begin{corollary}\label{bound on |B|}
Let $D_0$ be a Burdened link diagram. Then $$|B(D_0)| \leq 8\min\deg V - 2c(D_0) + n(D_0).$$
\end{corollary}

\begin{proof}
Let $(x_1, \dots, x_m)$ be a Balancing sequence, and $D_m$ the Balanced diagram obtained by smoothing all crossings in the sequence. Observe that then $s(D_m)=|A(D_m)|=|A(D_0)|=s(D_0)$, and consider the following: 
\begin{align*}
    c(D_0)-m &= c(D_m) \hspace{45pt} \text{ and since $D_m$ is Balanced, this is }\\
    & = 2(s(D_m)-1), \\
    & = 2(s(D_0)-1). \hspace{15pt} \text{ Rearrange to obtain}\\
    m &= c(D_0)-2(s(D_0)-1) \\
    &= 2(c(D_0)-s(D_0) + 1) - c(D_0)\\
    & = 4\min \deg V(D_0) - c(D_0).
\end{align*}

Then by Lemma \ref{|B(D_0)| leq 2m + n(D_0)},
$$|B(D_0)| \leq  2m + n(D_0) =  8\min\deg V(D_0) - 2c(D_0)+ n(D_0).$$
\end{proof}

\begin{theorem}\label{bound on max degree}
Let $D$ be a Burdened link diagram. Then $$\max\deg V \leq \frac{8\min \deg V + n(D) -1}{2},$$
and when $D$ is a \underline{knot} diagram, $$\max \deg V \leq 4\min\deg V.$$
\end{theorem}

\begin{proof}

Consider that the lowest possible degree term in the Kauffman bracket polynomial would be contributed by the all-$B$ state. (For more details of the Kauffman bracket polynomial, we direct the reader to \cite{Kauffman}.) Since our $D$ is positive, the all-$B$ state contribution to the Kauffman bracket polynomial is of degree $-3w-c-2(|B|-1)=-4c-2|B|+2$. Since no state can contribute a term of degree strictly less than that of the all-$B$ state, we know that the minimal degree of the Kauffman bracket polynomial is greater than or equal to $-4c-2|B|+2$. This means that for the Jones polynomial $V$, we have 

\begin{align*}
    \max \deg V & \leq (-1/4)(-4c-2|B|+2)\\
    & = \frac{2c+|B|-1}{2}\\
    &\leq \frac{2c+(8\min\deg V - 2c + n)-1}{2} \text{ ( by Cor. \ref{bound on |B|})}\\
    &= \frac{8\min \deg V + n -1}{2}.
\end{align*}

And if $D$ is a knot diagram, then $n(D)=1$ and so $\max \deg V \leq 4\min \deg V.$
\end{proof}

To remove the trouble of having to procure a Burdened diagram, we can leverage known results about positive fibered links. In \cite{Stoimenow}, as a corollary to \ref{V_1 tree}, we have:
\begin{corollary}(Stoimenow) \cite{Stoimenow} \label{V_1=0 iff fibered}
 For a positive link, $V_1=0$ if and only if $L$ is fibered.
\end{corollary}

(For more information about this theorem, about what it means for a link to be fibered, and for an overview of work done in these areas in particular having to do with state diagrams and knot positivity, we refer the reader to Stoimenow's work, and also to \cite{FKP13} and \cite{Fut13}.)

This means we could have given an  equivalent definition: a \textbf{Burdened diagram} is a (non-split) reduced, positive link diagram of a fibered positive link. 

With this observation, Theorem \ref{bound on max degree} gives us our main result:

\begin{corollary} \label{max deg V leq 4 min deg V }
 If $K$ is a fibered positive knot,
 then $\max \deg V \leq 4\min\deg V.$ 
\end{corollary}

\begin{proof}
Let $D$ be a positive diagram of $K$. Since $K$ is fibered, $V_1=0$ by \ref{V_1=0 iff fibered}. We can assume $D$ is reduced, so we can assume that it is a Burdened diagram, and then $\max \deg V \leq 4 \min \deg V$ by Theorem \ref{bound on max degree}.
\end{proof}

\subsection{Completing the positivity classification of knots with crossing number $\leq 12$}

Positivity is already known for all but seven knots with crossing number $\leq 12$. These remaining knots are $12_{n148}, 12_{n276}, 12_{n329}, 12_{n366}, 12_{n402}, 12_{n528},$ and $12_{n660}$. We now have the tools to show that all of these knots are not positive.

As noted in \cite{AT17}, \say{positivity unknown} for these knots means that it was unknown whether the knot \textit{or its mirror} is positive. As we have seen, it is a standard result that the minimum degree of the Jones polynomial of a positive knot is positive. Looking at the Jones polynomials given on KnotInfo \cite{KnotInfo} for our seven knots, we see that every exponent is negative. Therefore, the knots given are definitely not positive, but it remains to show that their mirrors could not have positive diagrams. The Jones polynomial of a knot's mirror is obtained by substituting $t^{-1}$ for $t$, so the mirrors of our seven knots have Jones polynomials with all positive exponents. These polynomials are:

\begin{align*}
    \text{Jones polynomial of } 12_{n148}!: & \hspace{7pt} t^3 + t^6 - 2t^7 + 3t^8 -3t^9 + 3t^{10} -3t^{11} + 2t^{12} - t^{13}\\
    12_{n276}!: & \hspace{7pt} t^3 + 2t^6 - 3t^7 + 4t^8 - 5t^9 + 4t^{10} - 4t^{11} + 3t^{12} - t^{13}\\
    12_{n329}!: & \hspace{7pt} t^3 + 2t^6 - 3t^7 + 3t^8 - 4t^9 + 4t^{10} - 3t^{11} + 2t^{12} - t^{13}\\
    12_{n366}!: & \hspace{7pt} t^3 - t^5 + 3t^6 - 4t^7 + 5t^8 -5t^9 + 5t^{10} - 4t^{11} + 2t^{12} - t^{13}\\
    12_{n402}!: & \hspace{7pt} t^3 - t^7 + 2t^8 - t^9 + 2t^{10} - 2t^{11} + t^{12} - t^{13}\\
    12_{n528}!: & \hspace{7pt} t^3 - t^5 + 4t^6 - 5t^7 + 6t^8 - 7t^9 + 6t^{10} - 5t^{11} + 3t^{12} - t^{13}\\
    12_{n660}!: & \hspace{7pt} t^3 - 2t^5 + 5t^6 - 6t^7 + 8t^8 - 8t^9 + 7t^{10} - 6t^{11} + 3t^{12} - t^{13}\\
\end{align*}

However, we also note that each knot is fibered, has $\min \deg V = 3$, and $\max \deg =13$. Since $13 > 4(3)=12$, our Corollary \ref{max deg V leq 4 min deg V } tells us that in fact none of these seven knots can be a positive knot. This completes classification of all knots of crossing number $\leq 12$ as positive or not positive.

\section{Proof of Theorem \ref{comp=Bcircles}}

This section will culminate in a proof of our key theorem from the previous section: 

\textbf{Theorem \ref{comp=Bcircles}} \textit{
Let $D$ be a Balanced link diagram. Then $n(D)=|B(D)|$.} \\

So, in this section we will always assume that $D$ is a Balanced diagram, and we will need to introduce some new definitions. 

\begin{definition}
Let $A$ and $A'$ be two $A$-circles that share crossings. Since $D$ is Balanced, they share exactly two crossings. 
Call these crossings a \textbf{matching pair}. 
\end{definition}

\begin{figure}[h]
    \centering
    \begin{tikzpicture}[every path/.style={thick}, every
node/.style={transform shape, knot crossing, inner sep=2pt}, knot gap = 9, scale =0.405]

    \begin{scope}[yshift=1.5cm]
      \draw[thick, -{Stealth}] (0.5,-0.5) -- (-1,1);
        \draw[thick, -{Stealth}] (-1,-1) -- (1,1);
        \draw[thick, knot] (1,-1) -- (-0.5,0.5);
        \draw[thick, knot] (-1,-1) -- (0.5,0.5);
    \end{scope}
    
    \begin{scope}[yshift=-1.5cm]
      \draw[thick, -{Stealth}] (0.5,-0.5) -- (-1,1);
        \draw[thick, -{Stealth}] (-1,-1) -- (1,1);
        \draw[thick, knot] (1,-1) -- (-0.5,0.5);
        \draw[thick, knot] (-1,-1) -- (0.5,0.5);
    \end{scope}
    
     \begin{scope}
    \node (boxtl) at (-3, 3) {};
    \node (boxtr) at (-1, 3) {};
    \node (boxbl) at (-3, -3) {};
    \node (boxbr) at (-1, -3) {};
    
    \draw [thick, green] (boxtl.center) -- (boxtr.center);
    \draw [thick, green] (boxtl.center) -- (boxbl.center);
    \draw [thick, green] (boxbl.center) -- (boxbr.center);
    \draw [thick, green] (boxbr.center) -- (boxtr.center);
    \end{scope}
    
    \begin{scope}[xshift=4cm]
    \node (boxtl) at (-3, 3) {};
    \node (boxtr) at (-1, 3) {};
    \node (boxbl) at (-3, -3) {};
    \node (boxbr) at (-1, -3) {};
    
    \draw [thick, green] (boxtl.center) -- (boxtr.center);
    \draw [thick, green] (boxtl.center) -- (boxbl.center);
    \draw [thick, green] (boxbl.center) -- (boxbr.center);
    \draw [thick, green] (boxbr.center) -- (boxtr.center);
    \end{scope}

\begin{scope}[yshift=1.5cm, xshift=8cm]
      \draw[thick, -{Stealth}, magenta] (0.5,-0.5) -- (-1,1);
      \draw [blue, thick, knot]  (1,-1) -- (-0.5,0.5);
      \draw [blue, thick, knot]  (1,-1) -- (0,0);
      \draw [magenta, thick, knot] (0,0) -- (-0.5, 0.5);
      
      \draw [thick, -{Stealth}, blue] (-1,-1) -- (1,1);
      \draw [blue, thick, knot] (-0.5,-0.5) -- (0.5,0.5);
      \draw [magenta, thick, knot] (-1,-1) -- (0,0);
      
    \end{scope}
    
    \begin{scope}[yshift=-1.5cm, xshift=8cm]
      \draw[thick, -{Stealth}, magenta] (0.5,-0.5) -- (-1,1);
      \draw [blue, thick, knot]  (1,-1) -- (-0.5,0.5);
      \draw [blue, thick, knot]  (1,-1) -- (0,0);
      \draw [magenta, thick, knot] (0,0) -- (-0.5, 0.5);
      
      \draw [thick, -{Stealth}, blue] (-1,-1) -- (1,1);
      \draw [blue, thick, knot] (-0.5,-0.5) -- (0.5,0.5);
      \draw [magenta, thick, knot] (-1,-1) -- (0,0);
      
      \node (btl) at (-1, 1) {};
      \node (tbl) at (-1, 2) {};
      \node (btr) at (1, 1) {};
      \node (tbr) at (1, 2) {};
      
      \draw [thick, magenta] (btl.center) .. controls (btl.4 north west) and (tbl.4 south west) .. (tbl.center) ;
      
       \draw [thick, blue] (btr.center) .. controls (btr.4 north east) and (tbr.4 south east) .. (tbr.center) ;
      
    \end{scope}

    \begin{scope}[yshift=0cm, xshift=8cm]
    \node (ttl) at (-1, 2.5) {};
    \node (ttr) at (1, 2.5) {};
    \node (tbl) at (-1, 0.5) {};
    \node (tbr) at (1, 0.5) {};
    
    \node (btl) at (-1, -0.5) {};
    \node (btr) at (1, -0.5) {};
    \node (bbl) at (-1, -2.5) {};
    \node (bbr) at (1, -2.5) {};
    \node (leftcenter) at (-4.5, -0.75) {};
    \node (rightcenter) at (4.5, -0.75) {};
    
    \draw [magenta] (bbl.center) .. controls (bbl.4 north west) and (leftcenter.4 north east) .. (ttl.center);
    
    \draw [blue] (bbr.center) .. controls (bbr.4 north east) and (rightcenter.4 north west) .. (ttr.center);

    \end{scope}
    
     \begin{scope}[xshift=8cm]
    \node (boxtl) at (-3, 3) {};
    \node (boxtr) at (-1, 3) {};
    \node (boxbl) at (-3, -3) {};
    \node (boxbr) at (-1, -3) {};
    
    \draw [thick, green] (boxtl.center) -- (boxtr.center);
    \draw [thick, green] (boxtl.center) -- (boxbl.center);
    \draw [thick, green] (boxbl.center) -- (boxbr.center);
    \draw [thick, green] (boxbr.center) -- (boxtr.center);
    \end{scope}
    
    \begin{scope}[xshift=12cm]
    \node (boxtl) at (-3, 3) {};
    \node (boxtr) at (-1, 3) {};
    \node (boxbl) at (-3, -3) {};
    \node (boxbr) at (-1, -3) {};
    
    \draw [thick, green] (boxtl.center) -- (boxtr.center);
    \draw [thick, green] (boxtl.center) -- (boxbl.center);
    \draw [thick, green] (boxbl.center) -- (boxbr.center);
    \draw [thick, green] (boxbr.center) -- (boxtr.center);
    \end{scope}
    
    
    \begin{scope}[yshift=-4cm]
    \draw [-{Stealth}, double] (0,0) -- (0,-1.5);
    \end{scope}
    
    \begin{scope}[xshift=8cm, yshift=-4cm]
    \draw [-{Stealth}, double] (0,0) -- (0,-1.5);
    \end{scope}
    

\begin{scope}[xshift=0cm, yshift=-9cm]
    \node (boxtl) at (-3, 3) {};
    \node (boxtr) at (-1, 3) {};
    \node (boxbl) at (-3, -3) {};
    \node (boxbr) at (-1, -3) {};
    
    \draw [thick, green] (boxtl.center) -- (boxtr.center);
    \draw [thick, green] (boxtl.center) -- (boxbl.center);
    \draw [thick, green] (boxbl.center) -- (boxbr.center);
    \draw [thick, green] (boxbr.center) -- (boxtr.center);
    \end{scope}
    
    \begin{scope}[xshift=4cm, yshift=-9cm]
    \node (boxtl) at (-3, 3) {};
    \node (boxtr) at (-1, 3) {};
    \node (boxbl) at (-3, -3) {};
    \node (boxbr) at (-1, -3) {};
    
    \draw [thick, green] (boxtl.center) -- (boxtr.center);
    \draw [thick, green] (boxtl.center) -- (boxbl.center);
    \draw [thick, green] (boxbl.center) -- (boxbr.center);
    \draw [thick, green] (boxbr.center) -- (boxtr.center);
    \end{scope}
    
      \begin{scope}[yshift=-7.5cm, xshift=0cm]
    \node (tl) at (-1, 1) {};
    \node (tr) at (1, 1) {};
    \node (bl) at (-1, -1) {};
    \node (br) at (1, -1) {};
    
    \draw [-{Stealth}] (bl.center) .. controls (bl.8 north east) and (tl.8 south east) .. (tl.center);
    \draw [-{Stealth}] (br.center) .. controls (br.8 north west) and (tr.8 south west) .. (tr.center);
    \end{scope}
    
    \begin{scope}[yshift=-10.5cm, xshift=0cm]
    \node (tl) at (-1, 1) {};
    \node (tr) at (1, 1) {};
    \node (bl) at (-1, -1) {};
    \node (br) at (1, -1) {};

    \draw [-{Stealth}] (bl.center) .. controls (bl.8 north east) and (tl.8 south east) .. (tl.center);
    \draw [-{Stealth}] (br.center) .. controls (br.8 north west) and (tr.8 south west) .. (tr.center);
    \end{scope}

    
    \begin{scope}[yshift=-9cm, xshift=8cm]
    \node (ttl) at (-1, 2.5) {};
    \node (ttr) at (1, 2.5) {};
    \node (tbl) at (-1, 0.5) {};
    \node (tbr) at (1, 0.5) {};
    
    \draw [-{Stealth}, magenta] (tbl.center) .. controls (tbl.8 north east) and (ttl.8 south east) .. (ttl.center);
    \draw [-{Stealth}, blue] (tbr.center) .. controls (tbr.8 north west) and (ttr.8 south west) .. (ttr.center);
    
    \node (btl) at (-1, -0.5) {};
    \node (btr) at (1, -0.5) {};
    \node (bbl) at (-1, -2.5) {};
    \node (bbr) at (1, -2.5) {};
    \node (leftcenter) at (-4.5, -0.75) {};
    \node (rightcenter) at (4.5, -0.75) {};

    \draw [-{Stealth}, magenta] (bbl.center) .. controls (bbl.8 north east) and (btl.8 south east) .. (btl.center);
    \draw [-{Stealth}, blue] (bbr.center) .. controls (bbr.8 north west) and (btr.8 south west) .. (btr.center);
    
    \draw [magenta] (bbl.center) .. controls (bbl.4 north west) and (leftcenter.4 north east) .. (ttl.center);
    
    \draw [blue] (bbr.center) .. controls (bbr.4 north east) and (rightcenter.4 north west) .. (ttr.center);
    
     \node (btl) at (-1, -0.5) {};
      \node (tbl) at (-1, 0.5) {};
      \node (btr) at (1, -0.5) {};
      \node (tbr) at (1, 0.5) {};
      
      \draw [thick, magenta] (btl.center) .. controls (btl.4 north west) and (tbl.4 south west) .. (tbl.center) ;
      
       \draw [thick, blue] (btr.center) .. controls (btr.4 north east) and (tbr.4 south east) .. (tbr.center) ;
    
    \end{scope}
    
      \begin{scope}[xshift=8cm, yshift=-9cm]
    \node (boxtl) at (-3, 3) {};
    \node (boxtr) at (-1, 3) {};
    \node (boxbl) at (-3, -3) {};
    \node (boxbr) at (-1, -3) {};
    
    \draw [thick, green] (boxtl.center) -- (boxtr.center);
    \draw [thick, green] (boxtl.center) -- (boxbl.center);
    \draw [thick, green] (boxbl.center) -- (boxbr.center);
    \draw [thick, green] (boxbr.center) -- (boxtr.center);
    \end{scope}
    
    \begin{scope}[xshift=12cm, yshift=-9cm]
    \node (boxtl) at (-3, 3) {};
    \node (boxtr) at (-1, 3) {};
    \node (boxbl) at (-3, -3) {};
    \node (boxbr) at (-1, -3) {};
    
    \draw [thick, green] (boxtl.center) -- (boxtr.center);
    \draw [thick, green] (boxtl.center) -- (boxbl.center);
    \draw [thick, green] (boxbl.center) -- (boxbr.center);
    \draw [thick, green] (boxbr.center) -- (boxtr.center);
    \end{scope}

    \end{tikzpicture}
    \caption{ When we smooth a matching pair, we disconnect the diagram}
    \label{smoothing matching pair disconnects the diagram}
\end{figure}
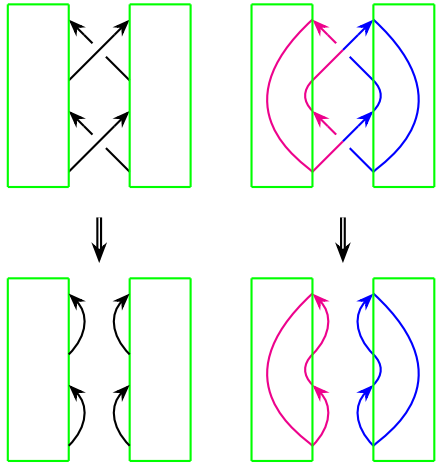

Since the (underlying reduced $A$-graph of this) diagram has a tree structure, if we smooth a matching pair we will disconnect the diagram. (See Figure \ref{smoothing matching pair disconnects the diagram}.) 

Observe that for a Balanced diagram $D$, we have a bijection between the set of matching pairs and the set of edges in the reduced $A$-state graph. Since this graph is a tree, if we start on an $A$-circle and pass through a matching pair (by following along a component or a $B$-circle), we can only return to the $A$-circle by first passing through that matching pair again. 

Then a $B$-circle must pass through a matching pair an even number of times, and also $B$-circles cannot cross themselves, which leaves us with four possibilities of how a $B$-circle could run through a matching pair (see Figure \ref{B-poss}).

Similarly, since components also straddle matching pairs and have a prescribed orientation, we are left with four possibilities of how a component could run through a matching pair (see Figure \ref{component poss}).


\begin{figure}[h]
    \centering
    
    \begin{tikzpicture}
    [every path/.style={thick}, every
node/.style={transform shape, knot crossing, inner sep=2pt}, knot gap = 9, scale = .44]
    
\begin{scope}[yshift=1.5cm, xshift=0cm]
      \draw[thick, orange] (0.5,-0.5) -- (-1,1);
      \draw [orange, thick, knot]  (1,-1) -- (-0.5,0.5);
      \draw [orange, thick, knot]  (1,-1) -- (0,0);
      \draw [orange, thick, knot] (0,0) -- (-0.5, 0.5);

      \draw [thick, orange] (-1,-1) -- (1,1);
      \draw [orange, thick, knot] (-0.5,-0.5) -- (0.5,0.5);
      \draw [orange, thick, knot] (-1,-1) -- (0,0);
      
      \draw [white, thick, knot] (-0.2, -0.2) -- (0.2, 0.2);
      \draw [orange, thick, knot] (-0.2, -0.2) -- (0.2, -0.2);
      \draw [orange, thick, knot] (-0.2, 0.2) -- (0.2, 0.2);

    \end{scope}
    
    \begin{scope}[yshift=-1.5cm, xshift=0cm]
      \draw[thick, orange] (0.5,-0.5) -- (-1,1);
      \draw [orange, thick, knot]  (1,-1) -- (-0.5,0.5);
      \draw [orange, thick, knot]  (1,-1) -- (0,0);
      \draw [orange, thick, knot] (0,0) -- (-0.5, 0.5);
      
      \draw [thick, orange] (-1,-1) -- (1,1);
      \draw [orange, thick, knot] (-0.5,-0.5) -- (0.5,0.5);
      \draw [orange, thick, knot] (-1,-1) -- (0,0);
      
      \draw [white, thick, knot] (-0.2, -0.2) -- (0.2, 0.2);
      \draw [orange, thick, knot] (-0.2, -0.2) -- (0.2, -0.2);
      \draw [orange, thick, knot] (-0.2, 0.2) -- (0.2, 0.2);
      
      \node (btl) at (-1, 1) {};
      \node (tbl) at (-1, 2) {};
      \node (btr) at (1, 1) {};
      \node (tbr) at (1, 2) {};
      
      \draw [thick, orange] (btl.center) .. controls (btl.4 north west) and (tbl.4 south west) .. (tbl.center) ;

    \end{scope}
    
    \begin{scope}[yshift=0cm, xshift=0cm]
    \node (ttl) at (-1, 2.5) {};
    \node (ttr) at (1, 2.5) {};
    \node (tbl) at (-1, 0.5) {};
    \node (tbr) at (1, 0.5) {};

    \node (btl) at (-1, -0.5) {};
    \node (btr) at (1, -0.5) {};
    \node (bbl) at (-1, -2.5) {};
    \node (bbr) at (1, -2.5) {};
    \node (leftcenter) at (-4.5, -0.75) {};
    \node (rightcenter) at (4.5, -0.75) {};
    
    \draw [thick, orange] (bbl.center) .. controls (bbl.4 north west) and (leftcenter.4 north east) .. (ttl.center);
    
    \draw [thick, orange] (tbr.center) .. controls (tbr.8 north east) and (ttr.8 south east) .. (ttr.center) ;
    
    \draw [orange] (bbr.center) .. controls (bbr.8 north east) and (btr.8 south east) .. (btr.center);
    \end{scope}

    \begin{scope}[xshift=0cm]
    \node (boxtl) at (-3, 3) {};
    \node (boxtr) at (-1, 3) {};
    \node (boxbl) at (-3, -3) {};
    \node (boxbr) at (-1, -3) {};
    
    \draw [thick, green] (boxtl.center) -- (boxtr.center);
    \draw [thick, green] (boxtl.center) -- (boxbl.center);
    \draw [thick, green] (boxbl.center) -- (boxbr.center);
    \draw [thick, green] (boxbr.center) -- (boxtr.center);
    \end{scope}
    
    \begin{scope}[xshift=4cm]
    \node (boxtl) at (-3, 3) {};
    \node (boxtr) at (-1, 3) {};
    \node (boxbl) at (-3, -3) {};
    \node (boxbr) at (-1, -3) {};
    
    \draw [thick, green] (boxtl.center) -- (boxtr.center);
    \draw [thick, green] (boxtl.center) -- (boxbl.center);
    \draw [thick, green] (boxbl.center) -- (boxbr.center);
    \draw [thick, green] (boxbr.center) -- (boxtr.center);
    \end{scope}
    
    
    \begin{scope}[yshift=1.5cm, xshift=8cm]
      \draw[thick, orange] (0.5,-0.5) -- (-1,1);
      \draw [orange, thick, knot]  (1,-1) -- (-0.5,0.5);
      \draw [orange, thick, knot]  (1,-1) -- (0,0);
      \draw [orange, thick, knot] (0,0) -- (-0.5, 0.5);
      
      \draw [thick, orange] (-1,-1) -- (1,1);
      \draw [orange, thick, knot] (-0.5,-0.5) -- (0.5,0.5);
      \draw [orange, thick, knot] (-1,-1) -- (0,0);
      
       \draw [white, thick, knot] (-0.2, -0.2) -- (0.2, 0.2);
      \draw [orange, thick, knot] (-0.2, -0.2) -- (0.2, -0.2);
      \draw [orange, thick, knot] (-0.2, 0.2) -- (0.2, 0.2);
      
    \end{scope}
    
    \begin{scope}[yshift=-1.5cm, xshift=8cm]
      \draw[thick,  orange] (0.5,-0.5) -- (-1,1);
      \draw [orange, thick, knot]  (1,-1) -- (-0.5,0.5);
      \draw [orange, thick, knot]  (1,-1) -- (0,0);
      \draw [orange, thick, knot] (0,0) -- (-0.5, 0.5);
      
      \draw [thick, orange] (-1,-1) -- (1,1);
      \draw [orange, thick, knot] (-0.5,-0.5) -- (0.5,0.5);
      \draw [orange, thick, knot] (-1,-1) -- (0,0);
      
       \draw [white, thick, knot] (-0.2, -0.2) -- (0.2, 0.2);
      \draw [orange, thick, knot] (-0.2, -0.2) -- (0.2, -0.2);
      \draw [orange, thick, knot] (-0.2, 0.2) -- (0.2, 0.2);
      
      \node (btl) at (-1, 1) {};
      \node (tbl) at (-1, 2) {};
      \node (btr) at (1, 1) {};
      \node (tbr) at (1, 2) {};

       \draw [thick, orange] (btr.center) .. controls (btr.4 north east) and (tbr.4 south east) .. (tbr.center) ;

    \end{scope}
    
    \begin{scope}[yshift=0cm, xshift=8cm]
    \node (ttl) at (-1, 2.5) {};
    \node (ttr) at (1, 2.5) {};
    \node (tbl) at (-1, 0.5) {};
    \node (tbr) at (1, 0.5) {};
    
    \node (btl) at (-1, -0.5) {};
    \node (btr) at (1, -0.5) {};
    \node (bbl) at (-1, -2.5) {};
    \node (bbr) at (1, -2.5) {};
    \node (leftcenter) at (-4.5, -0.75) {};
    \node (rightcenter) at (4.5, -0.75) {};

    \draw [orange] (tbl.center) .. controls (tbl.8 north west) and (ttl.8 south west) .. (ttl.center);
    
    \draw [orange] (bbl.center) .. controls (bbl.8 north west) and (btl.8 south west) .. (btl.center);
    
    \draw [orange] (bbr.center) .. controls (bbr.4 north east) and (rightcenter.4 north west) .. (ttr.center);
    \end{scope}

    \begin{scope}[xshift=8cm]
    \node (boxtl) at (-3, 3) {};
    \node (boxtr) at (-1, 3) {};
    \node (boxbl) at (-3, -3) {};
    \node (boxbr) at (-1, -3) {};
    
    \draw [thick, green] (boxtl.center) -- (boxtr.center);
    \draw [thick, green] (boxtl.center) -- (boxbl.center);
    \draw [thick, green] (boxbl.center) -- (boxbr.center);
    \draw [thick, green] (boxbr.center) -- (boxtr.center);
    \end{scope}
    
    \begin{scope}[xshift=12cm]
    \node (boxtl) at (-3, 3) {};
    \node (boxtr) at (-1, 3) {};
    \node (boxbl) at (-3, -3) {};
    \node (boxbr) at (-1, -3) {};
    
    \draw [thick, green] (boxtl.center) -- (boxtr.center);
    \draw [thick, green] (boxtl.center) -- (boxbl.center);
    \draw [thick, green] (boxbl.center) -- (boxbr.center);
    \draw [thick, green] (boxbr.center) -- (boxtr.center);
    \end{scope}
    
    
    \begin{scope}[yshift=1.5cm, xshift=16cm]
      \draw[thick, orange] (0.5,-0.5) -- (-1,1);
      \draw [orange, thick, knot]  (1,-1) -- (-0.5,0.5);
      \draw [orange, thick, knot]  (1,-1) -- (0,0);
      \draw [orange, thick, knot] (0,0) -- (-0.5, 0.5);
      
      \draw [thick,  orange] (-1,-1) -- (1,1);
      \draw [orange, thick, knot] (-0.5,-0.5) -- (0.5,0.5);
      \draw [orange, thick, knot] (-1,-1) -- (0,0);
      
       \draw [white, thick, knot] (-0.2, -0.2) -- (0.2, 0.2);
      \draw [orange, thick, knot] (-0.2, -0.2) -- (0.2, -0.2);
      \draw [orange, thick, knot] (-0.2, 0.2) -- (0.2, 0.2);
      
    \end{scope}
    
    \begin{scope}[yshift=-1.5cm, xshift=16cm]
      \draw[thick, cyan] (0.5,-0.5) -- (-1,1);
      \draw [cyan, thick, knot]  (1,-1) -- (-0.5,0.5);
      \draw [cyan, thick, knot]  (1,-1) -- (0,0);
      \draw [cyan, thick, knot] (0,0) -- (-0.5, 0.5);
      
      \draw [thick, cyan] (-1,-1) -- (1,1);
      \draw [cyan, thick, knot] (-0.5,-0.5) -- (0.5,0.5);
      \draw [cyan, thick, knot] (-1,-1) -- (0,0);
      
       \draw [white, thick, knot] (-0.2, -0.2) -- (0.2, 0.2);
      \draw [cyan, thick, knot] (-0.2, -0.2) -- (0.2, -0.2);
      \draw [cyan, thick, knot] (-0.2, 0.2) -- (0.2, 0.2);
      
      \node (btl) at (-1, 1) {};
      \node (tbl) at (-1, 2) {};
      \node (btr) at (1, 1) {};
      \node (tbr) at (1, 2) {};

    \end{scope}
    
    \begin{scope}[yshift=0cm, xshift=16cm]
    \node (ttl) at (-1, 2.5) {};
    \node (ttr) at (1, 2.5) {};
    \node (tbl) at (-1, 0.5) {};
    \node (tbr) at (1, 0.5) {};

    \node (btl) at (-1, -0.5) {};
    \node (btr) at (1, -0.5) {};
    \node (bbl) at (-1, -2.5) {};
    \node (bbr) at (1, -2.5) {};
    \node (leftcenter) at (-4.5, -0.75) {};
    \node (rightcenter) at (4.5, -0.75) {};
    
    \draw [thick, cyan] (bbl.center) .. controls (bbl.8 north west) and (btl.8 south west) .. (btl.center) ;
    
     \draw [thick, cyan] (bbr.center) .. controls (bbr.8 north east) and (btr.8 south east) .. (btr.center) ;
     
     \draw [thick, orange] (tbl.center) .. controls (tbl.8 north west) and (ttl.8 south west) .. (ttl.center) ;
    
     \draw [thick, orange] (tbr.center) .. controls (tbr.8 north east) and (ttr.8 south east) .. (ttr.center) ;
    
    \end{scope}

    \begin{scope}[xshift=16cm]
    \node (boxtl) at (-3, 3) {};
    \node (boxtr) at (-1, 3) {};
    \node (boxbl) at (-3, -3) {};
    \node (boxbr) at (-1, -3) {};
    
    \draw [thick, green] (boxtl.center) -- (boxtr.center);
    \draw [thick, green] (boxtl.center) -- (boxbl.center);
    \draw [thick, green] (boxbl.center) -- (boxbr.center);
    \draw [thick, green] (boxbr.center) -- (boxtr.center);
    \end{scope}
    
    \begin{scope}[xshift=20cm]
    \node (boxtl) at (-3, 3) {};
    \node (boxtr) at (-1, 3) {};
    \node (boxbl) at (-3, -3) {};
    \node (boxbr) at (-1, -3) {};
    
    \draw [thick, green] (boxtl.center) -- (boxtr.center);
    \draw [thick, green] (boxtl.center) -- (boxbl.center);
    \draw [thick, green] (boxbl.center) -- (boxbr.center);
    \draw [thick, green] (boxbr.center) -- (boxtr.center);
    \end{scope}
    
    
    \begin{scope}[yshift=1.5cm, xshift=24cm]
      \draw[thick,  orange] (0.5,-0.5) -- (-1,1);
      \draw [cyan, thick, knot]  (1,-1) -- (-0.5,0.5);
      \draw [cyan, thick, knot]  (1,-1) -- (0,0);
      \draw [orange, thick, knot] (0,0) -- (-0.5, 0.5);
      
      \draw [thick,  orange] (-1,-1) -- (1,1);
      \draw [orange, thick, knot] (-0.5,-0.5) -- (0.5,0.5);
      \draw [cyan, thick, knot] (-1,-1) -- (0,0);
      
       \draw [white, thick, knot] (-0.2, -0.2) -- (0.2, 0.2);
      \draw [cyan, thick, knot] (-0.2, -0.2) -- (0.2, -0.2);
      \draw [orange, thick, knot] (-0.2, 0.2) -- (0.2, 0.2);
      
    \end{scope}
    
    \begin{scope}[yshift=-1.5cm, xshift=24cm]
      \draw[thick,  cyan] (0.5,-0.5) -- (-1,1);
      \draw [orange, thick, knot]  (1,-1) -- (-0.5,0.5);
      \draw [orange, thick, knot]  (1,-1) -- (0,0);
      \draw [cyan, thick, knot] (0,0) -- (-0.5, 0.5);
      
      \draw [thick, cyan] (-1,-1) -- (1,1);
      \draw [cyan, thick, knot] (-0.5,-0.5) -- (0.5,0.5);
      \draw [orange, thick, knot] (-1,-1) -- (0,0);
      
       \draw [white, thick, knot] (-0.2, -0.2) -- (0.2, 0.2);
      \draw [orange, thick, knot] (-0.2, -0.2) -- (0.2, -0.2);
      \draw [cyan, thick, knot] (-0.2, 0.2) -- (0.2, 0.2);
      
      \node (btl) at (-1, 1) {};
      \node (tbl) at (-1, 2) {};
      \node (btr) at (1, 1) {};
      \node (tbr) at (1, 2) {};
      
      \draw [thick, cyan] (btl.center) .. controls (btl.4 north west) and (tbl.4 south west) .. (tbl.center) ;
      
       \draw [thick, cyan] (btr.center) .. controls (btr.4 north east) and (tbr.4 south east) .. (tbr.center) ;
      
    \end{scope}
    
    \begin{scope}[yshift=0cm, xshift=24cm]
    \node (ttl) at (-1, 2.5) {};
    \node (ttr) at (1, 2.5) {};
    \node (tbl) at (-1, 0.5) {};
    \node (tbr) at (1, 0.5) {};
    
    \node (btl) at (-1, -0.5) {};
    \node (btr) at (1, -0.5) {};
    \node (bbl) at (-1, -2.5) {};
    \node (bbr) at (1, -2.5) {};
    \node (leftcenter) at (-4.5, -0.75) {};
    \node (rightcenter) at (4.5, -0.75) {};
    
    \draw [orange] (bbl.center) .. controls (bbl.4 north west) and (leftcenter.4 north east) .. (ttl.center);
    
    \draw [orange] (bbr.center) .. controls (bbr.4 north east) and (rightcenter.4 north west) .. (ttr.center);
    \end{scope}

    \begin{scope}[xshift=24cm]
    \node (boxtl) at (-3, 3) {};
    \node (boxtr) at (-1, 3) {};
    \node (boxbl) at (-3, -3) {};
    \node (boxbr) at (-1, -3) {};
    
    \draw [thick, green] (boxtl.center) -- (boxtr.center);
    \draw [thick, green] (boxtl.center) -- (boxbl.center);
    \draw [thick, green] (boxbl.center) -- (boxbr.center);
    \draw [thick, green] (boxbr.center) -- (boxtr.center);
    \end{scope}
    
    \begin{scope}[xshift=28cm]
    \node (boxtl) at (-3, 3) {};
    \node (boxtr) at (-1, 3) {};
    \node (boxbl) at (-3, -3) {};
    \node (boxbr) at (-1, -3) {};
    
    \draw [thick, green] (boxtl.center) -- (boxtr.center);
    \draw [thick, green] (boxtl.center) -- (boxbl.center);
    \draw [thick, green] (boxbl.center) -- (boxbr.center);
    \draw [thick, green] (boxbr.center) -- (boxtr.center);
    \end{scope}
    \end{tikzpicture}

    \caption{ $B$-circle possibilities for a Matching Pair}
    \label{B-poss}
\end{figure}
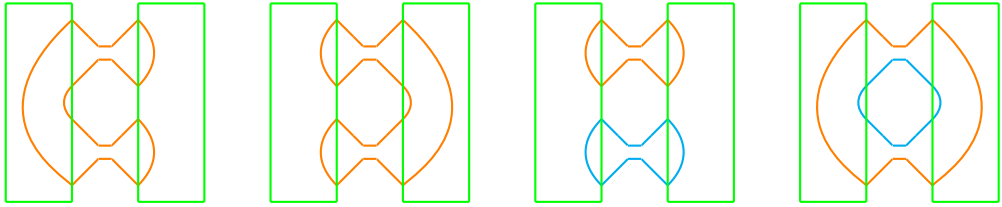


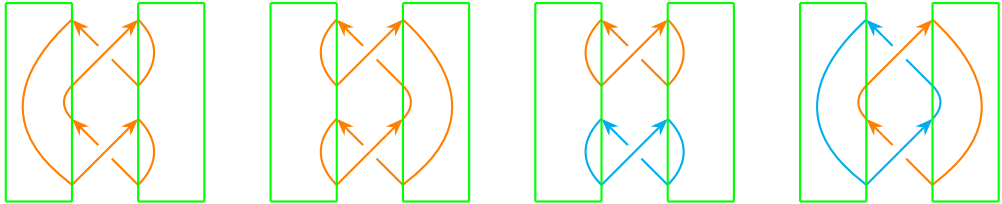
\begin{figure}[h]
    \centering
    
    \begin{tikzpicture}
    [every path/.style={thick}, every
node/.style={transform shape, knot crossing, inner sep=2pt}, knot gap = 9, scale = .44]
    
\begin{scope}[yshift=1.5cm, xshift=0cm]
      \draw[thick, -{Stealth}, orange] (0.5,-0.5) -- (-1,1);
      \draw [orange, thick, knot]  (1,-1) -- (-0.5,0.5);
      \draw [orange, thick, knot]  (1,-1) -- (0,0);
      \draw [orange, thick, knot] (0,0) -- (-0.5, 0.5);

      \draw [thick, -{Stealth}, orange] (-1,-1) -- (1,1);
      \draw [orange, thick, knot] (-0.5,-0.5) -- (0.5,0.5);
      \draw [orange, thick, knot] (-1,-1) -- (0,0);

    \end{scope}
    
    \begin{scope}[yshift=-1.5cm, xshift=0cm]
      \draw[thick, -{Stealth}, orange] (0.5,-0.5) -- (-1,1);
      \draw [orange, thick, knot]  (1,-1) -- (-0.5,0.5);
      \draw [orange, thick, knot]  (1,-1) -- (0,0);
      \draw [orange, thick, knot] (0,0) -- (-0.5, 0.5);
      
      \draw [thick, -{Stealth}, orange] (-1,-1) -- (1,1);
      \draw [orange, thick, knot] (-0.5,-0.5) -- (0.5,0.5);
      \draw [orange, thick, knot] (-1,-1) -- (0,0);
      
      \node (btl) at (-1, 1) {};
      \node (tbl) at (-1, 2) {};
      \node (btr) at (1, 1) {};
      \node (tbr) at (1, 2) {};
      
      \draw [thick, orange] (btl.center) .. controls (btl.4 north west) and (tbl.4 south west) .. (tbl.center) ;
      
      
    \end{scope}
    
    \begin{scope}[yshift=0cm, xshift=0cm]
    \node (ttl) at (-1, 2.5) {};
    \node (ttr) at (1, 2.5) {};
    \node (tbl) at (-1, 0.5) {};
    \node (tbr) at (1, 0.5) {};

    \node (btl) at (-1, -0.5) {};
    \node (btr) at (1, -0.5) {};
    \node (bbl) at (-1, -2.5) {};
    \node (bbr) at (1, -2.5) {};
    \node (leftcenter) at (-4.5, -0.75) {};
    \node (rightcenter) at (4.5, -0.75) {};
    
    \draw [thick, orange] (bbl.center) .. controls (bbl.4 north west) and (leftcenter.4 north east) .. (ttl.center);
    
    \draw [thick, orange] (tbr.center) .. controls (tbr.8 north east) and (ttr.8 south east) .. (ttr.center) ;
    
    \draw [orange] (bbr.center) .. controls (bbr.8 north east) and (btr.8 south east) .. (btr.center);
    \end{scope}

    \begin{scope}[xshift=0cm]
    \node (boxtl) at (-3, 3) {};
    \node (boxtr) at (-1, 3) {};
    \node (boxbl) at (-3, -3) {};
    \node (boxbr) at (-1, -3) {};
    
    \draw [thick, green] (boxtl.center) -- (boxtr.center);
    \draw [thick, green] (boxtl.center) -- (boxbl.center);
    \draw [thick, green] (boxbl.center) -- (boxbr.center);
    \draw [thick, green] (boxbr.center) -- (boxtr.center);
    \end{scope}
    
    \begin{scope}[xshift=4cm]
    \node (boxtl) at (-3, 3) {};
    \node (boxtr) at (-1, 3) {};
    \node (boxbl) at (-3, -3) {};
    \node (boxbr) at (-1, -3) {};
    
    \draw [thick, green] (boxtl.center) -- (boxtr.center);
    \draw [thick, green] (boxtl.center) -- (boxbl.center);
    \draw [thick, green] (boxbl.center) -- (boxbr.center);
    \draw [thick, green] (boxbr.center) -- (boxtr.center);
    \end{scope}
    
    
    \begin{scope}[yshift=1.5cm, xshift=8cm]
      \draw[thick, -{Stealth}, orange] (0.5,-0.5) -- (-1,1);
      \draw [orange, thick, knot]  (1,-1) -- (-0.5,0.5);
      \draw [orange, thick, knot]  (1,-1) -- (0,0);
      \draw [orange, thick, knot] (0,0) -- (-0.5, 0.5);
      
      \draw [thick, -{Stealth}, orange] (-1,-1) -- (1,1);
      \draw [orange, thick, knot] (-0.5,-0.5) -- (0.5,0.5);
      \draw [orange, thick, knot] (-1,-1) -- (0,0);
      
    \end{scope}
    
    \begin{scope}[yshift=-1.5cm, xshift=8cm]
      \draw[thick, -{Stealth}, orange] (0.5,-0.5) -- (-1,1);
      \draw [orange, thick, knot]  (1,-1) -- (-0.5,0.5);
      \draw [orange, thick, knot]  (1,-1) -- (0,0);
      \draw [orange, thick, knot] (0,0) -- (-0.5, 0.5);
      
      \draw [thick, -{Stealth}, orange] (-1,-1) -- (1,1);
      \draw [orange, thick, knot] (-0.5,-0.5) -- (0.5,0.5);
      \draw [orange, thick, knot] (-1,-1) -- (0,0);
      
      \node (btl) at (-1, 1) {};
      \node (tbl) at (-1, 2) {};
      \node (btr) at (1, 1) {};
      \node (tbr) at (1, 2) {};

       \draw [thick, orange] (btr.center) .. controls (btr.4 north east) and (tbr.4 south east) .. (tbr.center) ;

    \end{scope}
    
    \begin{scope}[yshift=0cm, xshift=8cm]
    \node (ttl) at (-1, 2.5) {};
    \node (ttr) at (1, 2.5) {};
    \node (tbl) at (-1, 0.5) {};
    \node (tbr) at (1, 0.5) {};

    \node (btl) at (-1, -0.5) {};
    \node (btr) at (1, -0.5) {};
    \node (bbl) at (-1, -2.5) {};
    \node (bbr) at (1, -2.5) {};
    \node (leftcenter) at (-4.5, -0.75) {};
    \node (rightcenter) at (4.5, -0.75) {};
    
    \draw [orange] (tbl.center) .. controls (tbl.8 north west) and (ttl.8 south west) .. (ttl.center);
    
    \draw [orange] (bbl.center) .. controls (bbl.8 north west) and (btl.8 south west) .. (btl.center);
    
    \draw [orange] (bbr.center) .. controls (bbr.4 north east) and (rightcenter.4 north west) .. (ttr.center);
    \end{scope}

    \begin{scope}[xshift=8cm]
    \node (boxtl) at (-3, 3) {};
    \node (boxtr) at (-1, 3) {};
    \node (boxbl) at (-3, -3) {};
    \node (boxbr) at (-1, -3) {};
    
    \draw [thick, green] (boxtl.center) -- (boxtr.center);
    \draw [thick, green] (boxtl.center) -- (boxbl.center);
    \draw [thick, green] (boxbl.center) -- (boxbr.center);
    \draw [thick, green] (boxbr.center) -- (boxtr.center);
    \end{scope}
    
    \begin{scope}[xshift=12cm]
    \node (boxtl) at (-3, 3) {};
    \node (boxtr) at (-1, 3) {};
    \node (boxbl) at (-3, -3) {};
    \node (boxbr) at (-1, -3) {};
    
    \draw [thick, green] (boxtl.center) -- (boxtr.center);
    \draw [thick, green] (boxtl.center) -- (boxbl.center);
    \draw [thick, green] (boxbl.center) -- (boxbr.center);
    \draw [thick, green] (boxbr.center) -- (boxtr.center);
    \end{scope}
    
    
    \begin{scope}[yshift=1.5cm, xshift=16cm]
      \draw[thick, -{Stealth}, orange] (0.5,-0.5) -- (-1,1);
      \draw [orange, thick, knot]  (1,-1) -- (-0.5,0.5);
      \draw [orange, thick, knot]  (1,-1) -- (0,0);
      \draw [orange, thick, knot] (0,0) -- (-0.5, 0.5);
      
      \draw [thick, -{Stealth}, orange] (-1,-1) -- (1,1);
      \draw [orange, thick, knot] (-0.5,-0.5) -- (0.5,0.5);
      \draw [orange, thick, knot] (-1,-1) -- (0,0);

    \end{scope}
    
    \begin{scope}[yshift=-1.5cm, xshift=16cm]
      \draw[thick, -{Stealth}, cyan] (0.5,-0.5) -- (-1,1);
      \draw [cyan, thick, knot]  (1,-1) -- (-0.5,0.5);
      \draw [cyan, thick, knot]  (1,-1) -- (0,0);
      \draw [cyan, thick, knot] (0,0) -- (-0.5, 0.5);
      
      \draw [thick, -{Stealth}, cyan] (-1,-1) -- (1,1);
      \draw [cyan, thick, knot] (-0.5,-0.5) -- (0.5,0.5);
      \draw [cyan, thick, knot] (-1,-1) -- (0,0);

      \node (btl) at (-1, 1) {};
      \node (tbl) at (-1, 2) {};
      \node (btr) at (1, 1) {};
      \node (tbr) at (1, 2) {};

    \end{scope}
    
    \begin{scope}[yshift=0cm, xshift=16cm]
    \node (ttl) at (-1, 2.5) {};
    \node (ttr) at (1, 2.5) {};
    \node (tbl) at (-1, 0.5) {};
    \node (tbr) at (1, 0.5) {};
    
    \node (btl) at (-1, -0.5) {};
    \node (btr) at (1, -0.5) {};
    \node (bbl) at (-1, -2.5) {};
    \node (bbr) at (1, -2.5) {};
    \node (leftcenter) at (-4.5, -0.75) {};
    \node (rightcenter) at (4.5, -0.75) {};
    
    \draw [thick, cyan] (bbl.center) .. controls (bbl.8 north west) and (btl.8 south west) .. (btl.center) ;
    
     \draw [thick, cyan] (bbr.center) .. controls (bbr.8 north east) and (btr.8 south east) .. (btr.center) ;
     
     \draw [thick, orange] (tbl.center) .. controls (tbl.8 north west) and (ttl.8 south west) .. (ttl.center) ;
    
     \draw [thick, orange] (tbr.center) .. controls (tbr.8 north east) and (ttr.8 south east) .. (ttr.center) ;
    
    \end{scope}

    \begin{scope}[xshift=16cm]
    \node (boxtl) at (-3, 3) {};
    \node (boxtr) at (-1, 3) {};
    \node (boxbl) at (-3, -3) {};
    \node (boxbr) at (-1, -3) {};
    
    \draw [thick, green] (boxtl.center) -- (boxtr.center);
    \draw [thick, green] (boxtl.center) -- (boxbl.center);
    \draw [thick, green] (boxbl.center) -- (boxbr.center);
    \draw [thick, green] (boxbr.center) -- (boxtr.center);
    \end{scope}
    
    \begin{scope}[xshift=20cm]
    \node (boxtl) at (-3, 3) {};
    \node (boxtr) at (-1, 3) {};
    \node (boxbl) at (-3, -3) {};
    \node (boxbr) at (-1, -3) {};
    
    \draw [thick, green] (boxtl.center) -- (boxtr.center);
    \draw [thick, green] (boxtl.center) -- (boxbl.center);
    \draw [thick, green] (boxbl.center) -- (boxbr.center);
    \draw [thick, green] (boxbr.center) -- (boxtr.center);
    \end{scope}
    
    
    \begin{scope}[yshift=1.5cm, xshift=24cm]
      \draw[thick, -{Stealth}, cyan] (0.5,-0.5) -- (-1,1);
      \draw [orange, thick, knot]  (1,-1) -- (-0.5,0.5);
      \draw [cyan, thick, knot]  (1,-1) -- (0,0);
      \draw [cyan, thick, knot] (0,0) -- (-0.5, 0.5);
      
      \draw [thick, -{Stealth}, orange] (-1,-1) -- (1,1);
      \draw [orange, thick, knot] (-0.5,-0.5) -- (0.5,0.5);
      \draw [orange, thick, knot] (-1,-1) -- (0,0);

    \end{scope}
    
    \begin{scope}[yshift=-1.5cm, xshift=24cm]
      \draw[thick, -{Stealth}, orange] (0.5,-0.5) -- (-1,1);
      \draw [cyan, thick, knot]  (1,-1) -- (-0.5,0.5);
      \draw [orange, thick, knot]  (1,-1) -- (0,0);
      \draw [orange, thick, knot] (0,0) -- (-0.5, 0.5);
      
      \draw [thick, -{Stealth}, cyan] (-1,-1) -- (1,1);
      \draw [cyan, thick, knot] (-0.5,-0.5) -- (0.5,0.5);
      \draw [cyan, thick, knot] (-1,-1) -- (0,0);

      \node (btl) at (-1, 1) {};
      \node (tbl) at (-1, 2) {};
      \node (btr) at (1, 1) {};
      \node (tbr) at (1, 2) {};
      
      \draw [thick, orange] (btl.center) .. controls (btl.4 north west) and (tbl.4 south west) .. (tbl.center) ;
      
       \draw [thick, cyan] (btr.center) .. controls (btr.4 north east) and (tbr.4 south east) .. (tbr.center) ;
      
    \end{scope}
    
    \begin{scope}[yshift=0cm, xshift=24cm]
    \node (ttl) at (-1, 2.5) {};
    \node (ttr) at (1, 2.5) {};
    \node (tbl) at (-1, 0.5) {};
    \node (tbr) at (1, 0.5) {};
    
    \node (btl) at (-1, -0.5) {};
    \node (btr) at (1, -0.5) {};
    \node (bbl) at (-1, -2.5) {};
    \node (bbr) at (1, -2.5) {};
    \node (leftcenter) at (-4.5, -0.75) {};
    \node (rightcenter) at (4.5, -0.75) {};

    \draw [cyan] (bbl.center) .. controls (bbl.4 north west) and (leftcenter.4 north east) .. (ttl.center);
    
    \draw [orange] (bbr.center) .. controls (bbr.4 north east) and (rightcenter.4 north west) .. (ttr.center);
    \end{scope}

    \begin{scope}[xshift=24cm]
    \node (boxtl) at (-3, 3) {};
    \node (boxtr) at (-1, 3) {};
    \node (boxbl) at (-3, -3) {};
    \node (boxbr) at (-1, -3) {};
    
    \draw [thick, green] (boxtl.center) -- (boxtr.center);
    \draw [thick, green] (boxtl.center) -- (boxbl.center);
    \draw [thick, green] (boxbl.center) -- (boxbr.center);
    \draw [thick, green] (boxbr.center) -- (boxtr.center);
    \end{scope}
    
    \begin{scope}[xshift=28cm]
    \node (boxtl) at (-3, 3) {};
    \node (boxtr) at (-1, 3) {};
    \node (boxbl) at (-3, -3) {};
    \node (boxbr) at (-1, -3) {};
    
    \draw [thick, green] (boxtl.center) -- (boxtr.center);
    \draw [thick, green] (boxtl.center) -- (boxbl.center);
    \draw [thick, green] (boxbl.center) -- (boxbr.center);
    \draw [thick, green] (boxbr.center) -- (boxtr.center);
    \end{scope}
    
    \end{tikzpicture}

     \caption{ Component possibilities for a Matching Pair}
    \label{component poss}
\end{figure}

\begin{definition}
We say that a matching pair is  \textbf{synchronized} if arcs of the matching pair that belong to the same $A$-circle are also part of the same $B$-circle if and only if they are part of the same component.
\end{definition}

 (So, for example, if a matching pair has the third possibility for $B$-circles shown in Figure \ref{B-poss}, then it is a synchronized matching pair if and only if it has the third possibility for components shown in Figure \ref{component poss}.)

In Figure \ref{A-circle, B-circle, component}, we see an example of a Balanced diagram in which \textit{every} matching pair is synchronized. In this example, we can observe that for any arcs $x$ and $y$ that are part of the blue $A$-circle, those two arcs are part of the same $B$-circle if and only if they are part of the same component. It turns out that this will always be the case when we have a Balanced diagram in which every matching pair is synchronized. While reading the proof of this, it may be helpful to look back at the diagram in Figure \ref{A-circle, B-circle, component}. 

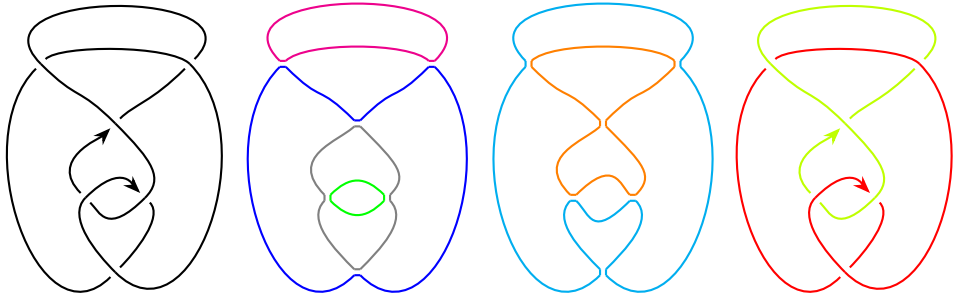
\begin{figure}
    \centering
    \begin{tikzpicture}[every path/.style={thick}, every
node/.style={transform shape, knot crossing, inner sep=2pt}, scale=0.99]
    \begin{scope}[xshift=0cm, scale=0.8]
    \node (tlo) at (-6.25,2.25) {};
    \node (tro) at (-3.75,2.25) {};
    \node (so) at (-5,1.25) {};
    \node (mlo) at (-5.5,0) {};
    \node (mro) at (-4.5, 0) {};
    \node (bo) at (-5,-1.25) {};
    \draw (bo.center) .. controls (bo.4 north west) and (mlo.4 south west) ..
    (mlo.center);
    \draw  (bo) .. controls (bo.4 north east) and (mro.4 south east).. (mro);
    \draw  [-{Stealth}] (mlo) .. controls (mlo.8 north west) and (so.3 south west) ..
    (so);
    \draw [-{Stealth}] (mlo.center) .. controls (mlo.8 north east) and (mro.2 north
    west) .. (mro);
    \draw  (mro.center) .. controls (mro.4 north east) and (so.8 south east) ..
    (so.center);
    \draw (mro.center) .. controls (mro.8 south west) and (mlo.3 south east) ..
    (mlo);
    \draw (so) .. controls (so.4 north east) and (tro.8 south west) .. (tro);
    \draw  (so.center) .. controls (so.8 north west) and (tlo.8 south
    east) .. (tlo.center);
    \draw (tlo.center) .. controls (tlo.16 north west) and (tro.16 north
    east) .. (tro);
    \draw (bo.center) .. controls (bo.16 south east) and (tro.16 south east) ..
    (tro.center);
    \draw (bo) .. controls (bo.16 south west) and (tlo.16 south
    west) .. (tlo);
    \draw (tlo) .. controls (tlo.4 north east) and (tro.4 north
    west) .. (tro.center);
    \end{scope}
    
    \begin{scope}[xshift=-0.75cm, scale=0.8]
    \node (tl) at (-1.25,2.25) {};
    \node (tltl) at (-1.3,2.3) {};
    \node (tlbl) at (-1.3,2.2) {};
    \node (tltr) at (-1.2,2.3) {};
    \node (tlbr) at (-1.2,2.2) {};
    \node (tr) at (1.25,2.25) {};
    \node (trtl) at (1.2,2.3) {};
    \node (trbl) at (1.2,2.2) {};
    \node (trtr) at (1.3,2.3) {};
    \node (trbr) at (1.3,2.2) {};
    \node (stl) at (-0.05,1.3) {};
    \node (str) at (0.05,1.3) {};
    \node (sbl) at (-0.05,1.2) {};
    \node (sbr) at (0.05,1.2) {};
    \node (ml) at (-0.5,0) {};
    \node (mltl) at (-0.55,0.05) {};
    \node (mlbl) at (-0.55,-0.05) {};
    \node (mltr) at (-0.45,0.05) {};
    \node (mlbr) at (-0.45,-0.05) {};
    \node (mr) at (0.5, 0) {};
    \node (mrtl) at (0.45, 0.05) {};
    \node (mrbl) at (0.45, -0.05) {};
    \node (mrtr) at (0.55, 0.05) {};
    \node (mrbr) at (0.55, -0.05) {};
    \node (b) at (0,-1.25) {};
    \node (btl) at (-0.05, -1.2) {};
    \node (bbl) at (-0.05, -1.3) {};
    \node (btr) at (0.05, -1.2) {};
    \node (bbr) at (0.05, -1.3) {};
    
    \draw [gray] (btl.center) .. controls (btl.4 north west) and (mlbl.4 south west) ..
    (mlbl.center);
    \draw [gray] (btr.center) .. controls (btr.4 north east) and (mrbr.4 south east)
    .. (mrbr.center);
    \draw [gray] (mltl.center) .. controls (mltl.8 north west) and (sbl.3 south west) ..
    (sbl.center);
    \draw [green] (mltr.center) .. controls (mltr.4 north east) and (mrtl.4 north
    west) .. (mrtl.center);
    \draw [gray] (mrtr.center) .. controls (mrtr.4 north east) and (sbr.8 south east) ..
    (sbr.center);
    \draw [green] (mrbl.center) .. controls (mrbl.4 south west) and (mlbr.4 south east) ..
    (mlbr.center);
    \draw [blue] (str.center) .. controls (str.8 north east) and (trbl.8 south west) .. (trbl.center);
    \draw [blue] (stl.center) .. controls (stl.8 north west) and (tlbr.8 south
    east) .. (tlbr.center);
    \draw [magenta] (tltl.center) .. controls (tltl.16 north west) and (trtr.16 north
    east) .. (trtr.center);
    \draw [blue] (bbr.center) .. controls (bbr.16 south east) and (trbr.16 south east) ..
    (trbr.center);
    \draw [blue] (bbl.center) .. controls (bbl.16 south west) and (tlbl.16 south
    west) .. (tlbl.center);
    \draw [magenta] (tltr.center) .. controls (tltr.4 north east) and (trtl.4 north
    west) .. (trtl.center);
    
    \draw [blue] (stl.center) -- (str.center) {};
    \draw [gray] (sbl.center) -- (sbr.center) {};
    \draw [gray] (mltl.center) -- (mlbl.center) {};
    \draw [green] (mltr.center) -- (mlbr.center) {};
    \draw [green] (mrtl.center) -- (mrbl.center) {};
    \draw [gray] (mrtr.center) -- (mrbr.center) {};
    \draw [magenta] (tltl.center) -- (tltr.center) {};
    \draw [blue] (tlbl.center) -- (tlbr.center) {};
    \draw [magenta] (trtl.center) -- (trtr.center) {};
    \draw [blue] (trbl.center) -- (trbr.center) {};
    \draw [gray] (btl.center) -- (btr.center) {};
    \draw [blue] (bbl.center) -- (bbr.center) {};
    \end{scope}
    

   \begin{scope}[xshift=2.55cm, scale=0.8]
    \node (tl) at (-1.25,2.25) {};
    \node (tltl) at (-1.3,2.3) {};
    \node (tlbl) at (-1.3,2.2) {};
    \node (tltr) at (-1.2,2.3) {};
    \node (tlbr) at (-1.2,2.2) {};
    \node (tr) at (1.25,2.25) {};
    \node (trtl) at (1.2,2.3) {};
    \node (trbl) at (1.2,2.2) {};
    \node (trtr) at (1.3,2.3) {};
    \node (trbr) at (1.3,2.2) {};
    \node (stl) at (-0.05,1.3) {};
    \node (str) at (0.05,1.3) {};
    \node (sbl) at (-0.05,1.2) {};
    \node (sbr) at (0.05,1.2) {};
    \node (ml) at (-0.5,0) {};
    \node (mltl) at (-0.55,0.05) {};
    \node (mlbl) at (-0.55,-0.05) {};
    \node (mltr) at (-0.45,0.05) {};
    \node (mlbr) at (-0.45,-0.05) {};
    \node (mr) at (0.5, 0) {};
    \node (mrtl) at (0.45, 0.05) {};
    \node (mrbl) at (0.45, -0.05) {};
    \node (mrtr) at (0.55, 0.05) {};
    \node (mrbr) at (0.55, -0.05) {};
    \node (b) at (0,-1.25) {};
    \node (btl) at (-0.05, -1.2) {};
    \node (bbl) at (-0.05, -1.3) {};
    \node (btr) at (0.05, -1.2) {};
    \node (bbr) at (0.05, -1.3) {};
    
    \draw [cyan] (btl.center) .. controls (btl.4 north west) and (mlbl.4 south west) ..
    (mlbl.center);
    \draw [cyan] (btr.center) .. controls (btr.4 north east) and (mrbr.4 south east)
    .. (mrbr.center);
    \draw [orange] (mltl.center) .. controls (mltl.8 north west) and (sbl.3 south west) ..
    (sbl.center);
    \draw [orange] (mltr.center) .. controls (mltr.8 north east) and (mrtl.2 north
    west) .. (mrtl.center);
    \draw [orange] (mrtr.center) .. controls (mrtr.4 north east) and (sbr.8 south east) ..
    (sbr.center);
    \draw [cyan] (mrbl.center) .. controls (mrbl.8 south west) and (mlbr.3 south east) ..
    (mlbr.center);
    \draw [orange] (str.center) .. controls (str.8 north east) and (trbl.8 south west) .. (trbl.center);
    \draw [orange] (stl.center) .. controls (stl.8 north west) and (tlbr.8 south
    east) .. (tlbr.center);
    \draw [cyan] (tltl.center) .. controls (tltl.16 north west) and (trtr.16 north
    east) .. (trtr.center);
    \draw [cyan] (bbr.center) .. controls (bbr.16 south east) and (trbr.16 south east) ..
    (trbr.center);
    \draw [cyan] (bbl.center) .. controls (bbl.16 south west) and (tlbl.16 south
    west) .. (tlbl.center);
    \draw [orange] (tltr.center) .. controls (tltr.4 north east) and (trtl.4 north
    west) .. (trtl.center);
    
    \draw [orange] (stl.center) -- (sbl.center) {};
    \draw [orange] (str.center) -- (sbr.center) {};
    \draw [orange] (mltl.center) -- (mltr.center) {};
    \draw [cyan] (mlbr.center) -- (mlbl.center) {};
    \draw [orange] (mrtl.center) -- (mrtr.center) {};
    \draw [cyan] (mrbr.center) -- (mrbl.center) {};
    \draw [cyan] (tltl.center) -- (tlbl.center) {};
    \draw [orange] (tlbr.center) -- (tltr.center) {};
    \draw [orange] (trtl.center) -- (trbl.center) {};
    \draw [cyan] (trbr.center) -- (trtr.center) {};
    \draw [cyan] (btl.center) -- (bbl.center) {};
    \draw [cyan] (btr.center) -- (bbr.center) {};
    \end{scope}
    
    \begin{scope}[xshift=9.8cm, scale=0.8]
        \node (tlo) at (-6.25,2.25) {};
        \node (tro) at (-3.75,2.25) {};
        \node (so) at (-5,1.25) {};
        \node (mlo) at (-5.5,0) {};
        \node (mro) at (-4.5, 0) {};
        \node (bo) at (-5,-1.25) {};
        \draw [red] (bo.center) .. controls (bo.4 north west) and (mlo.4 south west) ..
        (mlo.center);
        \draw  [red] (bo) .. controls (bo.4 north east) and (mro.4 south east).. (mro);
        \draw  [lime, -{Stealth}] (mlo) .. controls (mlo.8 north west) and (so.3 south west) ..
        (so);
        \draw [red, -{Stealth}] (mlo.center) .. controls (mlo.8 north east) and (mro.2 north
        west) .. (mro);
        \draw [lime] (mro.center) .. controls (mro.4 north east) and (so.8 south east) ..
        (so.center);
        \draw [lime] (mro.center) .. controls (mro.8 south west) and (mlo.3 south east) ..
        (mlo);
        \draw [lime] (so) .. controls (so.4 north east) and (tro.8 south west) .. (tro);
        \draw  [lime] (so.center) .. controls (so.8 north west) and (tlo.8 south
        east) .. (tlo.center);
        \draw [lime] (tlo.center) .. controls (tlo.16 north west) and (tro.16 north
        east) .. (tro);
        \draw [red] (bo.center) .. controls (bo.16 south east) and (tro.16 south east) ..
        (tro.center);
        \draw [red](bo) .. controls (bo.16 south west) and (tlo.16 south
        west) .. (tlo);
        \draw [red] (tlo) .. controls (tlo.4 north east) and (tro.4 north
        west) .. (tro.center);
    \end{scope}
\end{tikzpicture} 
\caption{ (Left to right:) A Balanced diagram, its $A$-circles, its $B$-circles, and its components}
\label{A-circle, B-circle, component}
\end{figure}

\begin{definition}
We say that a Balanced link diagram is \textbf{synchronized} if arcs on the same $A$-circle are also part of the same $B$-circle if and only if they are part of the same component
\end{definition}

\begin{lemma}
A Balanced diagram is synchronized if and only if all of its matching pairs are synchronized.
\label{synchronized means for same A:  same B iff same c}
\end{lemma} 

\begin{proof}

If a Balanced diagram is synchronized, then automatically all of its matching pairs are synchronized.

Now we would like to show that if all matching pairs are synchronized, then the Balanced diagram must also be synchronized. Let $x$ and $y$ be distinct arcs on $A$-circle $A$ that are part of the same component $C$. We would like to show that $x$ and $y$ must then also be part of the same $B$-circle.

We travel along $C$ from $x$ to get to $y$, and since these arcs are distinct, we must pass through matching pairs on our way. When we pass through a matching pair, we change $A$-circles. 
Since the underlying reduced $A$-state graph is a tree in which each edge uniquely corresponds to a matching pair of crossings in our diagram $D$, we must pass through the same matching pair again in order to return to $A$.
Since every matching pair is synchronized, when we return to $A$ we do so along an arc that is part of the same $B$-circle as $x$, by definition of a synchronized matching pair. Since this is true every time we leave $A$, when we arrive at arc $y$ we see that it also must be part of the same $B$-circle as arc $x$. 

Thus arcs on the same $A$-circle that are part of the same component must also be part of the same $B$-circle. 

By swapping the roles of component and $B$-circle in the paragraph above we get the reverse implication, and so we conclude that arcs on the same $A$-circle are part of the same $B$-circle if and only if they are part of the same component, and thus the diagram is synchronized by definition.
 
\end{proof}

\begin{lemma}\label{Every matching pair is synchronized}
In a Balanced diagram every matching pair is synchronized.
\end{lemma}

\begin{proof}

We proceed by induction on $|A|$, the number of $A$-circles. Our claim clearly holds in the base case $|A|=1$.

Now suppose our claim is true for any diagram with $|A|=k$ for some $k\in \mathbb{N}$. Consider a diagram $D$ for which $|A(D)|=k+1$. Choose any matching pair and smooth it. Because of the underlying tree structure, this disconnects our diagram and leaves us with two smaller diagrams. Call them $D'$ and $D''$. This is pictured in Figures \ref{B-possibilities} and  \ref{component possibilities}.

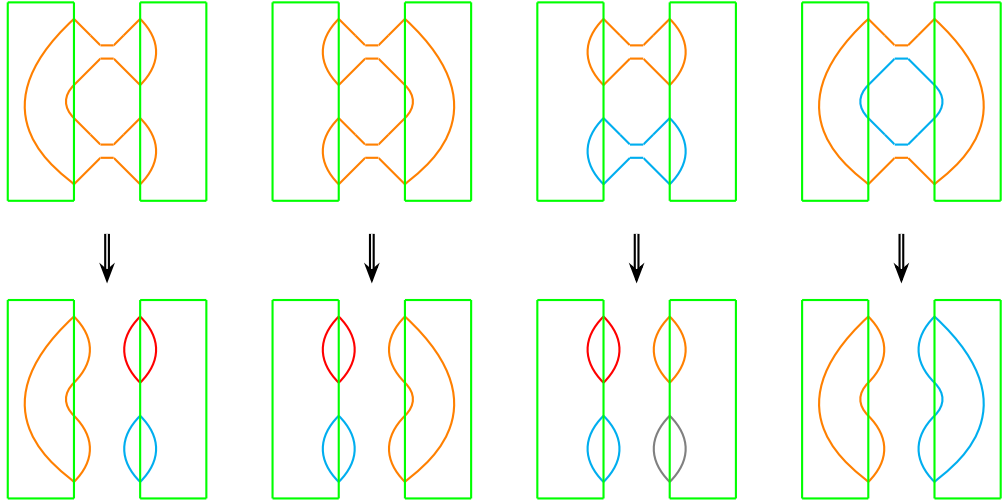
\begin{figure}[t]
    \centering
    
    \begin{tikzpicture}
    [every path/.style={thick}, every
node/.style={transform shape, knot crossing, inner sep=2pt}, knot gap = 9, scale = .44]
    
\begin{scope}[yshift=1.5cm, xshift=0cm]
      \draw[thick, orange] (0.5,-0.5) -- (-1,1);
      \draw [orange, thick, knot]  (1,-1) -- (-0.5,0.5);
      \draw [orange, thick, knot]  (1,-1) -- (0,0);
      \draw [orange, thick, knot] (0,0) -- (-0.5, 0.5);

      \draw [thick, orange] (-1,-1) -- (1,1);
      \draw [orange, thick, knot] (-0.5,-0.5) -- (0.5,0.5);
      \draw [orange, thick, knot] (-1,-1) -- (0,0);
      
      \draw [white, thick, knot] (-0.2, -0.2) -- (0.2, 0.2);
      \draw [orange, thick, knot] (-0.2, -0.2) -- (0.2, -0.2);
      \draw [orange, thick, knot] (-0.2, 0.2) -- (0.2, 0.2);
    \end{scope}
    
    \begin{scope}[yshift=-1.5cm, xshift=0cm]
      \draw[thick, orange] (0.5,-0.5) -- (-1,1);
      \draw [orange, thick, knot]  (1,-1) -- (-0.5,0.5);
      \draw [orange, thick, knot]  (1,-1) -- (0,0);
      \draw [orange, thick, knot] (0,0) -- (-0.5, 0.5);
      
      \draw [thick, orange] (-1,-1) -- (1,1);
      \draw [orange, thick, knot] (-0.5,-0.5) -- (0.5,0.5);
      \draw [orange, thick, knot] (-1,-1) -- (0,0);
      
      \draw [white, thick, knot] (-0.2, -0.2) -- (0.2, 0.2);
      \draw [orange, thick, knot] (-0.2, -0.2) -- (0.2, -0.2);
      \draw [orange, thick, knot] (-0.2, 0.2) -- (0.2, 0.2);
      
      \node (btl) at (-1, 1) {};
      \node (tbl) at (-1, 2) {};
      \node (btr) at (1, 1) {};
      \node (tbr) at (1, 2) {};
      
      \draw [thick, orange] (btl.center) .. controls (btl.4 north west) and (tbl.4 south west) .. (tbl.center) ;
      
      
    \end{scope}
    
    \begin{scope}[yshift=0cm, xshift=0cm]
    \node (ttl) at (-1, 2.5) {};
    \node (ttr) at (1, 2.5) {};
    \node (tbl) at (-1, 0.5) {};
    \node (tbr) at (1, 0.5) {};

    \node (btl) at (-1, -0.5) {};
    \node (btr) at (1, -0.5) {};
    \node (bbl) at (-1, -2.5) {};
    \node (bbr) at (1, -2.5) {};
    \node (leftcenter) at (-4.5, -0.75) {};
    \node (rightcenter) at (4.5, -0.75) {};

    \draw [thick, orange] (bbl.center) .. controls (bbl.4 north west) and (leftcenter.4 north east) .. (ttl.center);
    
    \draw [thick, orange] (tbr.center) .. controls (tbr.8 north east) and (ttr.8 south east) .. (ttr.center) ;
    
    \draw [orange] (bbr.center) .. controls (bbr.8 north east) and (btr.8 south east) .. (btr.center);
    \end{scope}

    \begin{scope}[xshift=0cm]
    \node (boxtl) at (-3, 3) {};
    \node (boxtr) at (-1, 3) {};
    \node (boxbl) at (-3, -3) {};
    \node (boxbr) at (-1, -3) {};
    
    \draw [thick, green] (boxtl.center) -- (boxtr.center);
    \draw [thick, green] (boxtl.center) -- (boxbl.center);
    \draw [thick, green] (boxbl.center) -- (boxbr.center);
    \draw [thick, green] (boxbr.center) -- (boxtr.center);
    \end{scope}
    
    \begin{scope}[xshift=4cm]
    \node (boxtl) at (-3, 3) {};
    \node (boxtr) at (-1, 3) {};
    \node (boxbl) at (-3, -3) {};
    \node (boxbr) at (-1, -3) {};
    
    \draw [thick, green] (boxtl.center) -- (boxtr.center);
    \draw [thick, green] (boxtl.center) -- (boxbl.center);
    \draw [thick, green] (boxbl.center) -- (boxbr.center);
    \draw [thick, green] (boxbr.center) -- (boxtr.center);
    \end{scope}
    
    
    \begin{scope}[yshift=1.5cm, xshift=8cm]
      \draw[thick, orange] (0.5,-0.5) -- (-1,1);
      \draw [orange, thick, knot]  (1,-1) -- (-0.5,0.5);
      \draw [orange, thick, knot]  (1,-1) -- (0,0);
      \draw [orange, thick, knot] (0,0) -- (-0.5, 0.5);
      
      \draw [thick, orange] (-1,-1) -- (1,1);
      \draw [orange, thick, knot] (-0.5,-0.5) -- (0.5,0.5);
      \draw [orange, thick, knot] (-1,-1) -- (0,0);
      
       \draw [white, thick, knot] (-0.2, -0.2) -- (0.2, 0.2);
      \draw [orange, thick, knot] (-0.2, -0.2) -- (0.2, -0.2);
      \draw [orange, thick, knot] (-0.2, 0.2) -- (0.2, 0.2);
      
    \end{scope}
    
    \begin{scope}[yshift=-1.5cm, xshift=8cm]
      \draw[thick,  orange] (0.5,-0.5) -- (-1,1);
      \draw [orange, thick, knot]  (1,-1) -- (-0.5,0.5);
      \draw [orange, thick, knot]  (1,-1) -- (0,0);
      \draw [orange, thick, knot] (0,0) -- (-0.5, 0.5);
      
      \draw [thick, orange] (-1,-1) -- (1,1);
      \draw [orange, thick, knot] (-0.5,-0.5) -- (0.5,0.5);
      \draw [orange, thick, knot] (-1,-1) -- (0,0);
      
       \draw [white, thick, knot] (-0.2, -0.2) -- (0.2, 0.2);
      \draw [orange, thick, knot] (-0.2, -0.2) -- (0.2, -0.2);
      \draw [orange, thick, knot] (-0.2, 0.2) -- (0.2, 0.2);
      
      \node (btl) at (-1, 1) {};
      \node (tbl) at (-1, 2) {};
      \node (btr) at (1, 1) {};
      \node (tbr) at (1, 2) {};

       \draw [thick, orange] (btr.center) .. controls (btr.4 north east) and (tbr.4 south east) .. (tbr.center) ;

    \end{scope}
    
    \begin{scope}[yshift=0cm, xshift=8cm]
    \node (ttl) at (-1, 2.5) {};
    \node (ttr) at (1, 2.5) {};
    \node (tbl) at (-1, 0.5) {};
    \node (tbr) at (1, 0.5) {};
    
    \node (btl) at (-1, -0.5) {};
    \node (btr) at (1, -0.5) {};
    \node (bbl) at (-1, -2.5) {};
    \node (bbr) at (1, -2.5) {};
    \node (leftcenter) at (-4.5, -0.75) {};
    \node (rightcenter) at (4.5, -0.75) {};

    \draw [orange] (tbl.center) .. controls (tbl.8 north west) and (ttl.8 south west) .. (ttl.center);
    
    \draw [orange] (bbl.center) .. controls (bbl.8 north west) and (btl.8 south west) .. (btl.center);
    
    \draw [orange] (bbr.center) .. controls (bbr.4 north east) and (rightcenter.4 north west) .. (ttr.center);
    \end{scope}

    \begin{scope}[xshift=8cm]
    \node (boxtl) at (-3, 3) {};
    \node (boxtr) at (-1, 3) {};
    \node (boxbl) at (-3, -3) {};
    \node (boxbr) at (-1, -3) {};
    
    \draw [thick, green] (boxtl.center) -- (boxtr.center);
    \draw [thick, green] (boxtl.center) -- (boxbl.center);
    \draw [thick, green] (boxbl.center) -- (boxbr.center);
    \draw [thick, green] (boxbr.center) -- (boxtr.center);
    \end{scope}
    
    \begin{scope}[xshift=12cm]
    \node (boxtl) at (-3, 3) {};
    \node (boxtr) at (-1, 3) {};
    \node (boxbl) at (-3, -3) {};
    \node (boxbr) at (-1, -3) {};
    
    \draw [thick, green] (boxtl.center) -- (boxtr.center);
    \draw [thick, green] (boxtl.center) -- (boxbl.center);
    \draw [thick, green] (boxbl.center) -- (boxbr.center);
    \draw [thick, green] (boxbr.center) -- (boxtr.center);
    \end{scope}
    
    
    \begin{scope}[yshift=1.5cm, xshift=16cm]
      \draw[thick, orange] (0.5,-0.5) -- (-1,1);
      \draw [orange, thick, knot]  (1,-1) -- (-0.5,0.5);
      \draw [orange, thick, knot]  (1,-1) -- (0,0);
      \draw [orange, thick, knot] (0,0) -- (-0.5, 0.5);
      
      \draw [thick,  orange] (-1,-1) -- (1,1);
      \draw [orange, thick, knot] (-0.5,-0.5) -- (0.5,0.5);
      \draw [orange, thick, knot] (-1,-1) -- (0,0);
      
       \draw [white, thick, knot] (-0.2, -0.2) -- (0.2, 0.2);
      \draw [orange, thick, knot] (-0.2, -0.2) -- (0.2, -0.2);
      \draw [orange, thick, knot] (-0.2, 0.2) -- (0.2, 0.2);
      
    \end{scope}
    
    \begin{scope}[yshift=-1.5cm, xshift=16cm]
      \draw[thick, cyan] (0.5,-0.5) -- (-1,1);
      \draw [cyan, thick, knot]  (1,-1) -- (-0.5,0.5);
      \draw [cyan, thick, knot]  (1,-1) -- (0,0);
      \draw [cyan, thick, knot] (0,0) -- (-0.5, 0.5);
      
      \draw [thick, cyan] (-1,-1) -- (1,1);
      \draw [cyan, thick, knot] (-0.5,-0.5) -- (0.5,0.5);
      \draw [cyan, thick, knot] (-1,-1) -- (0,0);
      
       \draw [white, thick, knot] (-0.2, -0.2) -- (0.2, 0.2);
      \draw [cyan, thick, knot] (-0.2, -0.2) -- (0.2, -0.2);
      \draw [cyan, thick, knot] (-0.2, 0.2) -- (0.2, 0.2);
      
      \node (btl) at (-1, 1) {};
      \node (tbl) at (-1, 2) {};
      \node (btr) at (1, 1) {};
      \node (tbr) at (1, 2) {};

    \end{scope}
    
    \begin{scope}[yshift=0cm, xshift=16cm]
    \node (ttl) at (-1, 2.5) {};
    \node (ttr) at (1, 2.5) {};
    \node (tbl) at (-1, 0.5) {};
    \node (tbr) at (1, 0.5) {};

    \node (btl) at (-1, -0.5) {};
    \node (btr) at (1, -0.5) {};
    \node (bbl) at (-1, -2.5) {};
    \node (bbr) at (1, -2.5) {};
    \node (leftcenter) at (-4.5, -0.75) {};
    \node (rightcenter) at (4.5, -0.75) {};
    
    \draw [thick, cyan] (bbl.center) .. controls (bbl.8 north west) and (btl.8 south west) .. (btl.center) ;
    
     \draw [thick, cyan] (bbr.center) .. controls (bbr.8 north east) and (btr.8 south east) .. (btr.center) ;
     
     \draw [thick, orange] (tbl.center) .. controls (tbl.8 north west) and (ttl.8 south west) .. (ttl.center) ;
    
     \draw [thick, orange] (tbr.center) .. controls (tbr.8 north east) and (ttr.8 south east) .. (ttr.center) ;
    
    \end{scope}

    \begin{scope}[xshift=16cm]
    \node (boxtl) at (-3, 3) {};
    \node (boxtr) at (-1, 3) {};
    \node (boxbl) at (-3, -3) {};
    \node (boxbr) at (-1, -3) {};
    
    \draw [thick, green] (boxtl.center) -- (boxtr.center);
    \draw [thick, green] (boxtl.center) -- (boxbl.center);
    \draw [thick, green] (boxbl.center) -- (boxbr.center);
    \draw [thick, green] (boxbr.center) -- (boxtr.center);
    \end{scope}
    
    \begin{scope}[xshift=20cm]
    \node (boxtl) at (-3, 3) {};
    \node (boxtr) at (-1, 3) {};
    \node (boxbl) at (-3, -3) {};
    \node (boxbr) at (-1, -3) {};
    
    \draw [thick, green] (boxtl.center) -- (boxtr.center);
    \draw [thick, green] (boxtl.center) -- (boxbl.center);
    \draw [thick, green] (boxbl.center) -- (boxbr.center);
    \draw [thick, green] (boxbr.center) -- (boxtr.center);
    \end{scope}
    
    
    \begin{scope}[yshift=1.5cm, xshift=24cm]
      \draw[thick,  orange] (0.5,-0.5) -- (-1,1);
      \draw [cyan, thick, knot]  (1,-1) -- (-0.5,0.5);
      \draw [cyan, thick, knot]  (1,-1) -- (0,0);
      \draw [orange, thick, knot] (0,0) -- (-0.5, 0.5);
      
      \draw [thick,  orange] (-1,-1) -- (1,1);
      \draw [orange, thick, knot] (-0.5,-0.5) -- (0.5,0.5);
      \draw [cyan, thick, knot] (-1,-1) -- (0,0);
      
       \draw [white, thick, knot] (-0.2, -0.2) -- (0.2, 0.2);
      \draw [cyan, thick, knot] (-0.2, -0.2) -- (0.2, -0.2);
      \draw [orange, thick, knot] (-0.2, 0.2) -- (0.2, 0.2);
      
    \end{scope}
    
    \begin{scope}[yshift=-1.5cm, xshift=24cm]
      \draw[thick,  cyan] (0.5,-0.5) -- (-1,1);
      \draw [orange, thick, knot]  (1,-1) -- (-0.5,0.5);
      \draw [orange, thick, knot]  (1,-1) -- (0,0);
      \draw [cyan, thick, knot] (0,0) -- (-0.5, 0.5);
      
      \draw [thick, cyan] (-1,-1) -- (1,1);
      \draw [cyan, thick, knot] (-0.5,-0.5) -- (0.5,0.5);
      \draw [orange, thick, knot] (-1,-1) -- (0,0);
      
       \draw [white, thick, knot] (-0.2, -0.2) -- (0.2, 0.2);
      \draw [orange, thick, knot] (-0.2, -0.2) -- (0.2, -0.2);
      \draw [cyan, thick, knot] (-0.2, 0.2) -- (0.2, 0.2);
      
      \node (btl) at (-1, 1) {};
      \node (tbl) at (-1, 2) {};
      \node (btr) at (1, 1) {};
      \node (tbr) at (1, 2) {};
      
      \draw [thick, cyan] (btl.center) .. controls (btl.4 north west) and (tbl.4 south west) .. (tbl.center) ;
      
       \draw [thick, cyan] (btr.center) .. controls (btr.4 north east) and (tbr.4 south east) .. (tbr.center) ;
      
    \end{scope}
    
    \begin{scope}[yshift=0cm, xshift=24cm]
    \node (ttl) at (-1, 2.5) {};
    \node (ttr) at (1, 2.5) {};
    \node (tbl) at (-1, 0.5) {};
    \node (tbr) at (1, 0.5) {};

    \node (btl) at (-1, -0.5) {};
    \node (btr) at (1, -0.5) {};
    \node (bbl) at (-1, -2.5) {};
    \node (bbr) at (1, -2.5) {};
    \node (leftcenter) at (-4.5, -0.75) {};
    \node (rightcenter) at (4.5, -0.75) {};
    
    \draw [orange] (bbl.center) .. controls (bbl.4 north west) and (leftcenter.4 north east) .. (ttl.center);
    
    \draw [orange] (bbr.center) .. controls (bbr.4 north east) and (rightcenter.4 north west) .. (ttr.center);
    \end{scope}

    \begin{scope}[xshift=24cm]
    \node (boxtl) at (-3, 3) {};
    \node (boxtr) at (-1, 3) {};
    \node (boxbl) at (-3, -3) {};
    \node (boxbr) at (-1, -3) {};
    
    \draw [thick, green] (boxtl.center) -- (boxtr.center);
    \draw [thick, green] (boxtl.center) -- (boxbl.center);
    \draw [thick, green] (boxbl.center) -- (boxbr.center);
    \draw [thick, green] (boxbr.center) -- (boxtr.center);
    \end{scope}
    
    \begin{scope}[xshift=28cm]
    \node (boxtl) at (-3, 3) {};
    \node (boxtr) at (-1, 3) {};
    \node (boxbl) at (-3, -3) {};
    \node (boxbr) at (-1, -3) {};
    
    \draw [thick, green] (boxtl.center) -- (boxtr.center);
    \draw [thick, green] (boxtl.center) -- (boxbl.center);
    \draw [thick, green] (boxbl.center) -- (boxbr.center);
    \draw [thick, green] (boxbr.center) -- (boxtr.center);
    \end{scope}

    
    \begin{scope}[yshift=-4cm]
    
        \begin{scope}
        \draw [-{Stealth}, double] (0,0) -- (0,-1.5);
        \end{scope}
        
        \begin{scope}[xshift=8cm]
        \draw [-{Stealth}, double] (0,0) -- (0,-1.5);
        \end{scope}
        
        \begin{scope}[xshift=16cm]
             \begin{scope}
            \draw [-{Stealth}, double] (0,0) -- (0,-1.5);
            \end{scope}
            
            \begin{scope}[xshift=8cm]
            \draw [-{Stealth}, double] (0,0) -- (0,-1.5);
            \end{scope}
        \end{scope}
        
    \end{scope}
    
    
    \begin{scope}[yshift=-9cm]
    
    \begin{scope}[yshift=-1.5cm, xshift=0cm]
      
      \node (btl) at (-1, 1) {};
      \node (tbl) at (-1, 2) {};
      \node (btr) at (1, 1) {};
      \node (tbr) at (1, 2) {};
      
      \draw [thick, orange] (btl.center) .. controls (btl.4 north west) and (tbl.4 south west) .. (tbl.center) ;
      
    \end{scope}
    
    \begin{scope}[yshift=0cm, xshift=0cm]
    \node (ttl) at (-1, 2.5) {};
    \node (ttr) at (1, 2.5) {};
    \node (tbl) at (-1, 0.5) {};
    \node (tbr) at (1, 0.5) {};

    \node (btl) at (-1, -0.5) {};
    \node (btr) at (1, -0.5) {};
    \node (bbl) at (-1, -2.5) {};
    \node (bbr) at (1, -2.5) {};
    \node (leftcenter) at (-4.5, -0.75) {};
    \node (rightcenter) at (4.5, -0.75) {};
    
    \draw [thick, orange] (bbl.center) .. controls (bbl.4 north west) and (leftcenter.4 north east) .. (ttl.center);
    
    \draw [thick, red] (tbr.center) .. controls (tbr.8 north east) and (ttr.8 south east) .. (ttr.center) ;
    
    \draw [thick, red] (tbr.center) .. controls (tbr.8 north west) and (ttr.8 south west) .. (ttr.center) ;
    
    \draw [cyan] (bbr.center) .. controls (bbr.8 north east) and (btr.8 south east) .. (btr.center);
    
    \draw [cyan] (bbr.center) .. controls (bbr.8 north west) and (btr.8 south west) .. (btr.center);
    
    \draw [thick, orange] (tbl.center) .. controls (tbl.8 north east) and (ttl.8 south east) .. (ttl.center) ;
    
    \draw [orange] (bbl.center) .. controls (bbl.8 north east) and (btl.8 south east) .. (btl.center);

    \end{scope}

    \begin{scope}[xshift=0cm]
    \node (boxtl) at (-3, 3) {};
    \node (boxtr) at (-1, 3) {};
    \node (boxbl) at (-3, -3) {};
    \node (boxbr) at (-1, -3) {};
    
    \draw [thick, green] (boxtl.center) -- (boxtr.center);
    \draw [thick, green] (boxtl.center) -- (boxbl.center);
    \draw [thick, green] (boxbl.center) -- (boxbr.center);
    \draw [thick, green] (boxbr.center) -- (boxtr.center);
    \end{scope}
    
    \begin{scope}[xshift=4cm]
    \node (boxtl) at (-3, 3) {};
    \node (boxtr) at (-1, 3) {};
    \node (boxbl) at (-3, -3) {};
    \node (boxbr) at (-1, -3) {};
    
    \draw [thick, green] (boxtl.center) -- (boxtr.center);
    \draw [thick, green] (boxtl.center) -- (boxbl.center);
    \draw [thick, green] (boxbl.center) -- (boxbr.center);
    \draw [thick, green] (boxbr.center) -- (boxtr.center);
    \end{scope}
    
    
    \begin{scope}[yshift=1.5cm, xshift=8cm]

    \end{scope}
    
    \begin{scope}[yshift=-1.5cm, xshift=8cm]

      \node (btl) at (-1, 1) {};
      \node (tbl) at (-1, 2) {};
      \node (btr) at (1, 1) {};
      \node (tbr) at (1, 2) {};

       \draw [thick, orange] (btr.center) .. controls (btr.4 north east) and (tbr.4 south east) .. (tbr.center) ;

    \end{scope}
    
    \begin{scope}[yshift=0cm, xshift=8cm]
    \node (ttl) at (-1, 2.5) {};
    \node (ttr) at (1, 2.5) {};
    \node (tbl) at (-1, 0.5) {};
    \node (tbr) at (1, 0.5) {};
    
    \node (btl) at (-1, -0.5) {};
    \node (btr) at (1, -0.5) {};
    \node (bbl) at (-1, -2.5) {};
    \node (bbr) at (1, -2.5) {};
    \node (leftcenter) at (-4.5, -0.75) {};
    \node (rightcenter) at (4.5, -0.75) {};

    \draw [thick, red] (tbl.center) .. controls (tbl.8 north east) and (ttl.8 south east) .. (ttl.center) ;
    
    \draw [thick, red] (tbl.center) .. controls (tbl.8 north west) and (ttl.8 south west) .. (ttl.center) ;
    
    \draw [cyan] (bbl.center) .. controls (bbl.8 north east) and (btl.8 south east) .. (btl.center);
    
    \draw [cyan] (bbl.center) .. controls (bbl.8 north west) and (btl.8 south west) .. (btl.center);
    
    \draw [orange] (bbr.center) .. controls (bbr.4 north east) and (rightcenter.4 north west) .. (ttr.center);
    
    \draw [thick, orange] (tbr.center) .. controls (tbr.8 north west) and (ttr.8 south west) .. (ttr.center) ;
    
    \draw [orange] (bbr.center) .. controls (bbr.8 north west) and (btr.8 south west) .. (btr.center);
    
    \end{scope}

    \begin{scope}[xshift=8cm]
    \node (boxtl) at (-3, 3) {};
    \node (boxtr) at (-1, 3) {};
    \node (boxbl) at (-3, -3) {};
    \node (boxbr) at (-1, -3) {};
    
    \draw [thick, green] (boxtl.center) -- (boxtr.center);
    \draw [thick, green] (boxtl.center) -- (boxbl.center);
    \draw [thick, green] (boxbl.center) -- (boxbr.center);
    \draw [thick, green] (boxbr.center) -- (boxtr.center);
    \end{scope}
    
    \begin{scope}[xshift=12cm]
    \node (boxtl) at (-3, 3) {};
    \node (boxtr) at (-1, 3) {};
    \node (boxbl) at (-3, -3) {};
    \node (boxbr) at (-1, -3) {};
    
    \draw [thick, green] (boxtl.center) -- (boxtr.center);
    \draw [thick, green] (boxtl.center) -- (boxbl.center);
    \draw [thick, green] (boxbl.center) -- (boxbr.center);
    \draw [thick, green] (boxbr.center) -- (boxtr.center);
    \end{scope}
    
    
    \begin{scope}[yshift=1.5cm, xshift=16cm]
   
    \end{scope}
    
    \begin{scope}[yshift=-1.5cm, xshift=16cm]

      \node (btl) at (-1, 1) {};
      \node (tbl) at (-1, 2) {};
      \node (btr) at (1, 1) {};
      \node (tbr) at (1, 2) {};

    \end{scope}
    
    \begin{scope}[yshift=0cm, xshift=16cm]
    \node (ttl) at (-1, 2.5) {};
    \node (ttr) at (1, 2.5) {};
    \node (tbl) at (-1, 0.5) {};
    \node (tbr) at (1, 0.5) {};

    \node (btl) at (-1, -0.5) {};
    \node (btr) at (1, -0.5) {};
    \node (bbl) at (-1, -2.5) {};
    \node (bbr) at (1, -2.5) {};
    \node (leftcenter) at (-4.5, -0.75) {};
    \node (rightcenter) at (4.5, -0.75) {};

     \draw [thick, red] (tbl.center) .. controls (tbl.8 north west) and (ttl.8 south west) .. (ttl.center) ;
     
     \draw [thick, red] (tbl.center) .. controls (tbl.8 north east) and (ttl.8 south east) .. (ttl.center) ;
    
     \draw [thick, orange] (tbr.center) .. controls (tbr.8 north east) and (ttr.8 south east) .. (ttr.center) ;
     
    \draw [thick, orange] (tbr.center) .. controls (tbr.8 north west) and (ttr.8 south west) .. (ttr.center) ;
    
    \draw [gray] (bbr.center) .. controls (bbr.8 north east) and (btr.8 south east) .. (btr.center);
    
    \draw [gray] (bbr.center) .. controls (bbr.8 north west) and (btr.8 south west) .. (btr.center);
    
     \draw [cyan] (bbl.center) .. controls (bbl.8 north east) and (btl.8 south east) .. (btl.center);
    
    \draw [cyan] (bbl.center) .. controls (bbl.8 north west) and (btl.8 south west) .. (btl.center);
    
    \end{scope}

    \begin{scope}[xshift=16cm]
    \node (boxtl) at (-3, 3) {};
    \node (boxtr) at (-1, 3) {};
    \node (boxbl) at (-3, -3) {};
    \node (boxbr) at (-1, -3) {};
    
    \draw [thick, green] (boxtl.center) -- (boxtr.center);
    \draw [thick, green] (boxtl.center) -- (boxbl.center);
    \draw [thick, green] (boxbl.center) -- (boxbr.center);
    \draw [thick, green] (boxbr.center) -- (boxtr.center);
    \end{scope}
    
    \begin{scope}[xshift=20cm]
    \node (boxtl) at (-3, 3) {};
    \node (boxtr) at (-1, 3) {};
    \node (boxbl) at (-3, -3) {};
    \node (boxbr) at (-1, -3) {};
    
    \draw [thick, green] (boxtl.center) -- (boxtr.center);
    \draw [thick, green] (boxtl.center) -- (boxbl.center);
    \draw [thick, green] (boxbl.center) -- (boxbr.center);
    \draw [thick, green] (boxbr.center) -- (boxtr.center);
    \end{scope}
    

    \begin{scope}[yshift=-1.5cm, xshift=24cm]

      \node (btl) at (-1, 1) {};
      \node (tbl) at (-1, 2) {};
      \node (btr) at (1, 1) {};
      \node (tbr) at (1, 2) {};
      
      \draw [thick, orange] (btl.center) .. controls (btl.4 north west) and (tbl.4 south west) .. (tbl.center) ;
      
       \draw [thick, cyan] (btr.center) .. controls (btr.4 north east) and (tbr.4 south east) .. (tbr.center) ;
      
    \end{scope}
    
    \begin{scope}[yshift=0cm, xshift=24cm]
    \node (ttl) at (-1, 2.5) {};
    \node (ttr) at (1, 2.5) {};
    \node (tbl) at (-1, 0.5) {};
    \node (tbr) at (1, 0.5) {};

    \node (btl) at (-1, -0.5) {};
    \node (btr) at (1, -0.5) {};
    \node (bbl) at (-1, -2.5) {};
    \node (bbr) at (1, -2.5) {};
    \node (leftcenter) at (-4.5, -0.75) {};
    \node (rightcenter) at (4.5, -0.75) {};

    \draw [orange] (bbl.center) .. controls (bbl.4 north west) and (leftcenter.4 north east) .. (ttl.center);
    
    \draw [cyan] (bbr.center) .. controls (bbr.4 north east) and (rightcenter.4 north west) .. (ttr.center);

    \draw [thick, cyan] (tbr.center) .. controls (tbr.8 north west) and (ttr.8 south west) .. (ttr.center) ;

    \draw [cyan] (bbr.center) .. controls (bbr.8 north west) and (btr.8 south west) .. (btr.center);
    
     \draw [orange] (bbl.center) .. controls (bbl.8 north east) and (btl.8 south east) .. (btl.center);

    \draw [orange] (tbl.center) .. controls (tbl.8 north east) and (ttl.8 south east) .. (ttl.center);

    \end{scope}

    \begin{scope}[xshift=24cm]
    \node (boxtl) at (-3, 3) {};
    \node (boxtr) at (-1, 3) {};
    \node (boxbl) at (-3, -3) {};
    \node (boxbr) at (-1, -3) {};
    
    \draw [thick, green] (boxtl.center) -- (boxtr.center);
    \draw [thick, green] (boxtl.center) -- (boxbl.center);
    \draw [thick, green] (boxbl.center) -- (boxbr.center);
    \draw [thick, green] (boxbr.center) -- (boxtr.center);
    \end{scope}
    
    \begin{scope}[xshift=28cm]
    \node (boxtl) at (-3, 3) {};
    \node (boxtr) at (-1, 3) {};
    \node (boxbl) at (-3, -3) {};
    \node (boxbr) at (-1, -3) {};
    
    \draw [thick, green] (boxtl.center) -- (boxtr.center);
    \draw [thick, green] (boxtl.center) -- (boxbl.center);
    \draw [thick, green] (boxbl.center) -- (boxbr.center);
    \draw [thick, green] (boxbr.center) -- (boxtr.center);
    \end{scope}
    
    \end{scope}
    
    \end{tikzpicture}

    \caption{ \textbf{Top row:} $B$-circle possibilities for a Matching pair; \hspace{37pt} 
    \textbf{Bottom row:} $B$-circles in the two new smaller diagrams that result from smoothing the Matching pair}
    \label{B-possibilities}
    \vspace{16pt}
\end{figure}


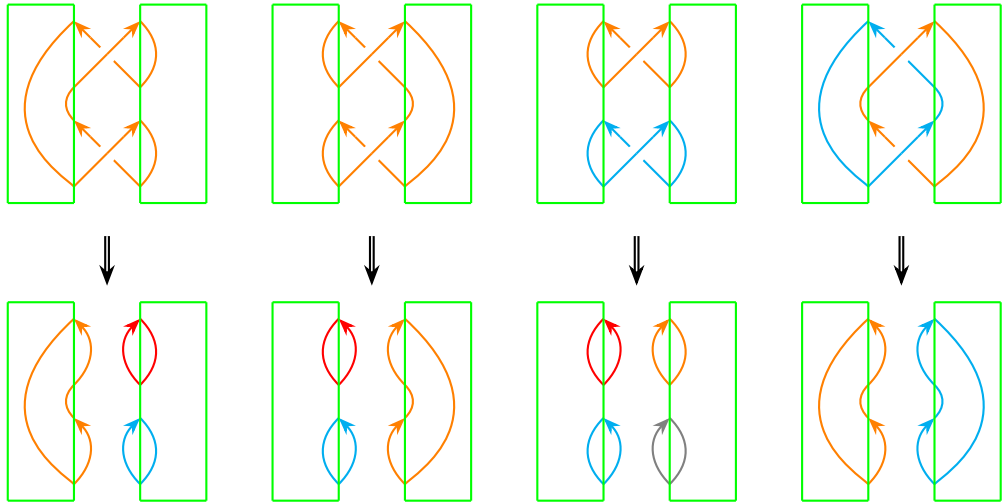
\begin{figure}
    \centering
    
    \begin{tikzpicture}
    [every path/.style={thick}, every
node/.style={transform shape, knot crossing, inner sep=2pt}, knot gap = 9, scale = .44]
    
\begin{scope}[yshift=1.5cm, xshift=0cm]
      \draw[thick, -{Stealth}, orange] (0.5,-0.5) -- (-1,1);
      \draw [orange, thick, knot]  (1,-1) -- (-0.5,0.5);
      \draw [orange, thick, knot]  (1,-1) -- (0,0);
      \draw [orange, thick, knot] (0,0) -- (-0.5, 0.5);

      \draw [thick, -{Stealth}, orange] (-1,-1) -- (1,1);
      \draw [orange, thick, knot] (-0.5,-0.5) -- (0.5,0.5);
      \draw [orange, thick, knot] (-1,-1) -- (0,0);

    \end{scope}
    
    \begin{scope}[yshift=-1.5cm, xshift=0cm]
      \draw[thick, -{Stealth}, orange] (0.5,-0.5) -- (-1,1);
      \draw [orange, thick, knot]  (1,-1) -- (-0.5,0.5);
      \draw [orange, thick, knot]  (1,-1) -- (0,0);
      \draw [orange, thick, knot] (0,0) -- (-0.5, 0.5);
      
      \draw [thick, -{Stealth}, orange] (-1,-1) -- (1,1);
      \draw [orange, thick, knot] (-0.5,-0.5) -- (0.5,0.5);
      \draw [orange, thick, knot] (-1,-1) -- (0,0);
      
      \node (btl) at (-1, 1) {};
      \node (tbl) at (-1, 2) {};
      \node (btr) at (1, 1) {};
      \node (tbr) at (1, 2) {};
      
      \draw [thick, orange] (btl.center) .. controls (btl.4 north west) and (tbl.4 south west) .. (tbl.center) ;

    \end{scope}
    
    \begin{scope}[yshift=0cm, xshift=0cm]
    \node (ttl) at (-1, 2.5) {};
    \node (ttr) at (1, 2.5) {};
    \node (tbl) at (-1, 0.5) {};
    \node (tbr) at (1, 0.5) {};

    \node (btl) at (-1, -0.5) {};
    \node (btr) at (1, -0.5) {};
    \node (bbl) at (-1, -2.5) {};
    \node (bbr) at (1, -2.5) {};
    \node (leftcenter) at (-4.5, -0.75) {};
    \node (rightcenter) at (4.5, -0.75) {};
    
    \draw [thick, orange] (bbl.center) .. controls (bbl.4 north west) and (leftcenter.4 north east) .. (ttl.center);
    
    \draw [thick, orange] (tbr.center) .. controls (tbr.8 north east) and (ttr.8 south east) .. (ttr.center) ;
    
    \draw [orange] (bbr.center) .. controls (bbr.8 north east) and (btr.8 south east) .. (btr.center);
    \end{scope}

    \begin{scope}[xshift=0cm]
    \node (boxtl) at (-3, 3) {};
    \node (boxtr) at (-1, 3) {};
    \node (boxbl) at (-3, -3) {};
    \node (boxbr) at (-1, -3) {};
    
    \draw [thick, green] (boxtl.center) -- (boxtr.center);
    \draw [thick, green] (boxtl.center) -- (boxbl.center);
    \draw [thick, green] (boxbl.center) -- (boxbr.center);
    \draw [thick, green] (boxbr.center) -- (boxtr.center);
    \end{scope}
    
    \begin{scope}[xshift=4cm]
    \node (boxtl) at (-3, 3) {};
    \node (boxtr) at (-1, 3) {};
    \node (boxbl) at (-3, -3) {};
    \node (boxbr) at (-1, -3) {};
    
    \draw [thick, green] (boxtl.center) -- (boxtr.center);
    \draw [thick, green] (boxtl.center) -- (boxbl.center);
    \draw [thick, green] (boxbl.center) -- (boxbr.center);
    \draw [thick, green] (boxbr.center) -- (boxtr.center);
    \end{scope}
    
    
    \begin{scope}[yshift=1.5cm, xshift=8cm]
      \draw[thick, -{Stealth}, orange] (0.5,-0.5) -- (-1,1);
      \draw [orange, thick, knot]  (1,-1) -- (-0.5,0.5);
      \draw [orange, thick, knot]  (1,-1) -- (0,0);
      \draw [orange, thick, knot] (0,0) -- (-0.5, 0.5);
      
      \draw [thick, -{Stealth}, orange] (-1,-1) -- (1,1);
      \draw [orange, thick, knot] (-0.5,-0.5) -- (0.5,0.5);
      \draw [orange, thick, knot] (-1,-1) -- (0,0);
      
    \end{scope}
    
    \begin{scope}[yshift=-1.5cm, xshift=8cm]
      \draw[thick, -{Stealth}, orange] (0.5,-0.5) -- (-1,1);
      \draw [orange, thick, knot]  (1,-1) -- (-0.5,0.5);
      \draw [orange, thick, knot]  (1,-1) -- (0,0);
      \draw [orange, thick, knot] (0,0) -- (-0.5, 0.5);
      
      \draw [thick, -{Stealth}, orange] (-1,-1) -- (1,1);
      \draw [orange, thick, knot] (-0.5,-0.5) -- (0.5,0.5);
      \draw [orange, thick, knot] (-1,-1) -- (0,0);
      
      \node (btl) at (-1, 1) {};
      \node (tbl) at (-1, 2) {};
      \node (btr) at (1, 1) {};
      \node (tbr) at (1, 2) {};

       \draw [thick, orange] (btr.center) .. controls (btr.4 north east) and (tbr.4 south east) .. (tbr.center) ;

    \end{scope}
    
    \begin{scope}[yshift=0cm, xshift=8cm]
    \node (ttl) at (-1, 2.5) {};
    \node (ttr) at (1, 2.5) {};
    \node (tbl) at (-1, 0.5) {};
    \node (tbr) at (1, 0.5) {};

    \node (btl) at (-1, -0.5) {};
    \node (btr) at (1, -0.5) {};
    \node (bbl) at (-1, -2.5) {};
    \node (bbr) at (1, -2.5) {};
    \node (leftcenter) at (-4.5, -0.75) {};
    \node (rightcenter) at (4.5, -0.75) {};

    \draw [orange] (tbl.center) .. controls (tbl.8 north west) and (ttl.8 south west) .. (ttl.center);
    
    \draw [orange] (bbl.center) .. controls (bbl.8 north west) and (btl.8 south west) .. (btl.center);

    \draw [orange] (bbr.center) .. controls (bbr.4 north east) and (rightcenter.4 north west) .. (ttr.center);
    \end{scope}

    \begin{scope}[xshift=8cm]
    \node (boxtl) at (-3, 3) {};
    \node (boxtr) at (-1, 3) {};
    \node (boxbl) at (-3, -3) {};
    \node (boxbr) at (-1, -3) {};
    
    \draw [thick, green] (boxtl.center) -- (boxtr.center);
    \draw [thick, green] (boxtl.center) -- (boxbl.center);
    \draw [thick, green] (boxbl.center) -- (boxbr.center);
    \draw [thick, green] (boxbr.center) -- (boxtr.center);
    \end{scope}
    
    \begin{scope}[xshift=12cm]
    \node (boxtl) at (-3, 3) {};
    \node (boxtr) at (-1, 3) {};
    \node (boxbl) at (-3, -3) {};
    \node (boxbr) at (-1, -3) {};
    
    \draw [thick, green] (boxtl.center) -- (boxtr.center);
    \draw [thick, green] (boxtl.center) -- (boxbl.center);
    \draw [thick, green] (boxbl.center) -- (boxbr.center);
    \draw [thick, green] (boxbr.center) -- (boxtr.center);
    \end{scope}
    
    
    \begin{scope}[yshift=1.5cm, xshift=16cm]
      \draw[thick, -{Stealth}, orange] (0.5,-0.5) -- (-1,1);
      \draw [orange, thick, knot]  (1,-1) -- (-0.5,0.5);
      \draw [orange, thick, knot]  (1,-1) -- (0,0);
      \draw [orange, thick, knot] (0,0) -- (-0.5, 0.5);
      
      \draw [thick, -{Stealth}, orange] (-1,-1) -- (1,1);
      \draw [orange, thick, knot] (-0.5,-0.5) -- (0.5,0.5);
      \draw [orange, thick, knot] (-1,-1) -- (0,0);

    \end{scope}
    
    \begin{scope}[yshift=-1.5cm, xshift=16cm]
      \draw[thick, -{Stealth}, cyan] (0.5,-0.5) -- (-1,1);
      \draw [cyan, thick, knot]  (1,-1) -- (-0.5,0.5);
      \draw [cyan, thick, knot]  (1,-1) -- (0,0);
      \draw [cyan, thick, knot] (0,0) -- (-0.5, 0.5);
      
      \draw [thick, -{Stealth}, cyan] (-1,-1) -- (1,1);
      \draw [cyan, thick, knot] (-0.5,-0.5) -- (0.5,0.5);
      \draw [cyan, thick, knot] (-1,-1) -- (0,0);

      \node (btl) at (-1, 1) {};
      \node (tbl) at (-1, 2) {};
      \node (btr) at (1, 1) {};
      \node (tbr) at (1, 2) {};

    \end{scope}
    
    \begin{scope}[yshift=0cm, xshift=16cm]
    \node (ttl) at (-1, 2.5) {};
    \node (ttr) at (1, 2.5) {};
    \node (tbl) at (-1, 0.5) {};
    \node (tbr) at (1, 0.5) {};

    \node (btl) at (-1, -0.5) {};
    \node (btr) at (1, -0.5) {};
    \node (bbl) at (-1, -2.5) {};
    \node (bbr) at (1, -2.5) {};
    \node (leftcenter) at (-4.5, -0.75) {};
    \node (rightcenter) at (4.5, -0.75) {};
    
    \draw [thick, cyan] (bbl.center) .. controls (bbl.8 north west) and (btl.8 south west) .. (btl.center) ;
    
     \draw [thick, cyan] (bbr.center) .. controls (bbr.8 north east) and (btr.8 south east) .. (btr.center) ;
     
     \draw [thick, orange] (tbl.center) .. controls (tbl.8 north west) and (ttl.8 south west) .. (ttl.center) ;
    
     \draw [thick, orange] (tbr.center) .. controls (tbr.8 north east) and (ttr.8 south east) .. (ttr.center) ;
    
    \end{scope}

    \begin{scope}[xshift=16cm]
    \node (boxtl) at (-3, 3) {};
    \node (boxtr) at (-1, 3) {};
    \node (boxbl) at (-3, -3) {};
    \node (boxbr) at (-1, -3) {};
    
    \draw [thick, green] (boxtl.center) -- (boxtr.center);
    \draw [thick, green] (boxtl.center) -- (boxbl.center);
    \draw [thick, green] (boxbl.center) -- (boxbr.center);
    \draw [thick, green] (boxbr.center) -- (boxtr.center);
    \end{scope}
    
    \begin{scope}[xshift=20cm]
    \node (boxtl) at (-3, 3) {};
    \node (boxtr) at (-1, 3) {};
    \node (boxbl) at (-3, -3) {};
    \node (boxbr) at (-1, -3) {};
    
    \draw [thick, green] (boxtl.center) -- (boxtr.center);
    \draw [thick, green] (boxtl.center) -- (boxbl.center);
    \draw [thick, green] (boxbl.center) -- (boxbr.center);
    \draw [thick, green] (boxbr.center) -- (boxtr.center);
    \end{scope}
    
    
    \begin{scope}[yshift=1.5cm, xshift=24cm]
      \draw[thick, -{Stealth}, cyan] (0.5,-0.5) -- (-1,1);
      \draw [orange, thick, knot]  (1,-1) -- (-0.5,0.5);
      \draw [cyan, thick, knot]  (1,-1) -- (0,0);
      \draw [cyan, thick, knot] (0,0) -- (-0.5, 0.5);
      
      \draw [thick, -{Stealth}, orange] (-1,-1) -- (1,1);
      \draw [orange, thick, knot] (-0.5,-0.5) -- (0.5,0.5);
      \draw [orange, thick, knot] (-1,-1) -- (0,0);

    \end{scope}
    
    \begin{scope}[yshift=-1.5cm, xshift=24cm]
      \draw[thick, -{Stealth}, orange] (0.5,-0.5) -- (-1,1);
      \draw [cyan, thick, knot]  (1,-1) -- (-0.5,0.5);
      \draw [orange, thick, knot]  (1,-1) -- (0,0);
      \draw [orange, thick, knot] (0,0) -- (-0.5, 0.5);
      
      \draw [thick, -{Stealth}, cyan] (-1,-1) -- (1,1);
      \draw [cyan, thick, knot] (-0.5,-0.5) -- (0.5,0.5);
      \draw [cyan, thick, knot] (-1,-1) -- (0,0);

      \node (btl) at (-1, 1) {};
      \node (tbl) at (-1, 2) {};
      \node (btr) at (1, 1) {};
      \node (tbr) at (1, 2) {};
      
      \draw [thick, orange] (btl.center) .. controls (btl.4 north west) and (tbl.4 south west) .. (tbl.center) ;
      
       \draw [thick, cyan] (btr.center) .. controls (btr.4 north east) and (tbr.4 south east) .. (tbr.center) ;
      
    \end{scope}
    
    \begin{scope}[yshift=0cm, xshift=24cm]
    \node (ttl) at (-1, 2.5) {};
    \node (ttr) at (1, 2.5) {};
    \node (tbl) at (-1, 0.5) {};
    \node (tbr) at (1, 0.5) {};
    
    \node (btl) at (-1, -0.5) {};
    \node (btr) at (1, -0.5) {};
    \node (bbl) at (-1, -2.5) {};
    \node (bbr) at (1, -2.5) {};
    \node (leftcenter) at (-4.5, -0.75) {};
    \node (rightcenter) at (4.5, -0.75) {};

    \draw [cyan] (bbl.center) .. controls (bbl.4 north west) and (leftcenter.4 north east) .. (ttl.center);
    
    \draw [orange] (bbr.center) .. controls (bbr.4 north east) and (rightcenter.4 north west) .. (ttr.center);
    \end{scope}

    \begin{scope}[xshift=24cm]
    \node (boxtl) at (-3, 3) {};
    \node (boxtr) at (-1, 3) {};
    \node (boxbl) at (-3, -3) {};
    \node (boxbr) at (-1, -3) {};
    
    \draw [thick, green] (boxtl.center) -- (boxtr.center);
    \draw [thick, green] (boxtl.center) -- (boxbl.center);
    \draw [thick, green] (boxbl.center) -- (boxbr.center);
    \draw [thick, green] (boxbr.center) -- (boxtr.center);
    \end{scope}
    
    \begin{scope}[xshift=28cm]
    \node (boxtl) at (-3, 3) {};
    \node (boxtr) at (-1, 3) {};
    \node (boxbl) at (-3, -3) {};
    \node (boxbr) at (-1, -3) {};
    
    \draw [thick, green] (boxtl.center) -- (boxtr.center);
    \draw [thick, green] (boxtl.center) -- (boxbl.center);
    \draw [thick, green] (boxbl.center) -- (boxbr.center);
    \draw [thick, green] (boxbr.center) -- (boxtr.center);
    \end{scope}

    
    \begin{scope}[yshift=-4cm]
    
        \begin{scope}
        \draw [-{Stealth}, double] (0,0) -- (0,-1.5);
        \end{scope}
        
        \begin{scope}[xshift=8cm]
        \draw [-{Stealth}, double] (0,0) -- (0,-1.5);
        \end{scope}
        
        \begin{scope}[xshift=16cm]
             \begin{scope}
            \draw [-{Stealth}, double] (0,0) -- (0,-1.5);
            \end{scope}
            
            \begin{scope}[xshift=8cm]
            \draw [-{Stealth}, double] (0,0) -- (0,-1.5);
            \end{scope}
        \end{scope}
        
    \end{scope}
    
    
    \begin{scope}[yshift=-9cm]
    
    \begin{scope}[yshift=-1.5cm, xshift=0cm]
      
      \node (btl) at (-1, 1) {};
      \node (tbl) at (-1, 2) {};
      \node (btr) at (1, 1) {};
      \node (tbr) at (1, 2) {};
      
      \draw [thick, orange] (btl.center) .. controls (btl.4 north west) and (tbl.4 south west) .. (tbl.center) ;
      
    \end{scope}
    
    \begin{scope}[yshift=0cm, xshift=0cm]
    \node (ttl) at (-1, 2.5) {};
    \node (ttr) at (1, 2.5) {};
    \node (tbl) at (-1, 0.5) {};
    \node (tbr) at (1, 0.5) {};

    \node (btl) at (-1, -0.5) {};
    \node (btr) at (1, -0.5) {};
    \node (bbl) at (-1, -2.5) {};
    \node (bbr) at (1, -2.5) {};
    \node (leftcenter) at (-4.5, -0.75) {};
    \node (rightcenter) at (4.5, -0.75) {};
    
    \draw [thick, orange] (bbl.center) .. controls (bbl.4 north west) and (leftcenter.4 north east) .. (ttl.center);
    
    \draw [thick, red] (tbr.center) .. controls (tbr.8 north east) and (ttr.8 south east) .. (ttr.center) ;
    
    \draw [thick, red, -{Stealth}] (tbr.center) .. controls (tbr.8 north west) and (ttr.8 south west) .. (ttr.center) ;
    
    \draw [cyan] (bbr.center) .. controls (bbr.8 north east) and (btr.8 south east) .. (btr.center);
    
    \draw [cyan, -{Stealth}] (bbr.center) .. controls (bbr.8 north west) and (btr.8 south west) .. (btr.center);
    
    \draw [thick, orange, -{Stealth}] (tbl.center) .. controls (tbl.8 north east) and (ttl.8 south east) .. (ttl.center) ;
    
    \draw [orange, -{Stealth}] (bbl.center) .. controls (bbl.8 north east) and (btl.8 south east) .. (btl.center);

    \end{scope}

    \begin{scope}[xshift=0cm]
    \node (boxtl) at (-3, 3) {};
    \node (boxtr) at (-1, 3) {};
    \node (boxbl) at (-3, -3) {};
    \node (boxbr) at (-1, -3) {};
    
    \draw [thick, green] (boxtl.center) -- (boxtr.center);
    \draw [thick, green] (boxtl.center) -- (boxbl.center);
    \draw [thick, green] (boxbl.center) -- (boxbr.center);
    \draw [thick, green] (boxbr.center) -- (boxtr.center);
    \end{scope}
    
    \begin{scope}[xshift=4cm]
    \node (boxtl) at (-3, 3) {};
    \node (boxtr) at (-1, 3) {};
    \node (boxbl) at (-3, -3) {};
    \node (boxbr) at (-1, -3) {};
    
    \draw [thick, green] (boxtl.center) -- (boxtr.center);
    \draw [thick, green] (boxtl.center) -- (boxbl.center);
    \draw [thick, green] (boxbl.center) -- (boxbr.center);
    \draw [thick, green] (boxbr.center) -- (boxtr.center);
    \end{scope}
    

    \begin{scope}[yshift=-1.5cm, xshift=8cm]

      \node (btl) at (-1, 1) {};
      \node (tbl) at (-1, 2) {};
      \node (btr) at (1, 1) {};
      \node (tbr) at (1, 2) {};

       \draw [thick, orange] (btr.center) .. controls (btr.4 north east) and (tbr.4 south east) .. (tbr.center) ;

    \end{scope}
    
    \begin{scope}[yshift=0cm, xshift=8cm]
    \node (ttl) at (-1, 2.5) {};
    \node (ttr) at (1, 2.5) {};
    \node (tbl) at (-1, 0.5) {};
    \node (tbr) at (1, 0.5) {};
    
    \node (btl) at (-1, -0.5) {};
    \node (btr) at (1, -0.5) {};
    \node (bbl) at (-1, -2.5) {};
    \node (bbr) at (1, -2.5) {};
    \node (leftcenter) at (-4.5, -0.75) {};
    \node (rightcenter) at (4.5, -0.75) {};

    \draw [thick, red, -{Stealth}] (tbl.center) .. controls (tbl.8 north east) and (ttl.8 south east) .. (ttl.center) ;
    
    \draw [thick, red] (tbl.center) .. controls (tbl.8 north west) and (ttl.8 south west) .. (ttl.center) ;
    
    \draw [cyan, -{Stealth}] (bbl.center) .. controls (bbl.8 north east) and (btl.8 south east) .. (btl.center);
    
    \draw [cyan] (bbl.center) .. controls (bbl.8 north west) and (btl.8 south west) .. (btl.center);
    
    \draw [orange] (bbr.center) .. controls (bbr.4 north east) and (rightcenter.4 north west) .. (ttr.center);
    
    \draw [thick, orange, -{Stealth}] (tbr.center) .. controls (tbr.8 north west) and (ttr.8 south west) .. (ttr.center) ;
    
    \draw [orange, -{Stealth}] (bbr.center) .. controls (bbr.8 north west) and (btr.8 south west) .. (btr.center);
    
    \end{scope}

    \begin{scope}[xshift=8cm]
    \node (boxtl) at (-3, 3) {};
    \node (boxtr) at (-1, 3) {};
    \node (boxbl) at (-3, -3) {};
    \node (boxbr) at (-1, -3) {};
    
    \draw [thick, green] (boxtl.center) -- (boxtr.center);
    \draw [thick, green] (boxtl.center) -- (boxbl.center);
    \draw [thick, green] (boxbl.center) -- (boxbr.center);
    \draw [thick, green] (boxbr.center) -- (boxtr.center);
    \end{scope}
    
    \begin{scope}[xshift=12cm]
    \node (boxtl) at (-3, 3) {};
    \node (boxtr) at (-1, 3) {};
    \node (boxbl) at (-3, -3) {};
    \node (boxbr) at (-1, -3) {};
    
    \draw [thick, green] (boxtl.center) -- (boxtr.center);
    \draw [thick, green] (boxtl.center) -- (boxbl.center);
    \draw [thick, green] (boxbl.center) -- (boxbr.center);
    \draw [thick, green] (boxbr.center) -- (boxtr.center);
    \end{scope}
    

    \begin{scope}[yshift=-1.5cm, xshift=16cm]

      \node (btl) at (-1, 1) {};
      \node (tbl) at (-1, 2) {};
      \node (btr) at (1, 1) {};
      \node (tbr) at (1, 2) {};

    \end{scope}
    
    \begin{scope}[yshift=0cm, xshift=16cm]
    \node (ttl) at (-1, 2.5) {};
    \node (ttr) at (1, 2.5) {};
    \node (tbl) at (-1, 0.5) {};
    \node (tbr) at (1, 0.5) {};

    \node (btl) at (-1, -0.5) {};
    \node (btr) at (1, -0.5) {};
    \node (bbl) at (-1, -2.5) {};
    \node (bbr) at (1, -2.5) {};
    \node (leftcenter) at (-4.5, -0.75) {};
    \node (rightcenter) at (4.5, -0.75) {};
     
     \draw [thick, red] (tbl.center) .. controls (tbl.8 north west) and (ttl.8 south west) .. (ttl.center) ;
     
     \draw [thick, red, -{Stealth}] (tbl.center) .. controls (tbl.8 north east) and (ttl.8 south east) .. (ttl.center) ;
    
     \draw [thick, orange] (tbr.center) .. controls (tbr.8 north east) and (ttr.8 south east) .. (ttr.center) ;
     
    \draw [thick, orange, -{Stealth}] (tbr.center) .. controls (tbr.8 north west) and (ttr.8 south west) .. (ttr.center) ;
    
    \draw [gray] (bbr.center) .. controls (bbr.8 north east) and (btr.8 south east) .. (btr.center);
    
    \draw [gray, -{Stealth}] (bbr.center) .. controls (bbr.8 north west) and (btr.8 south west) .. (btr.center);
    
     \draw [cyan, -{Stealth}] (bbl.center) .. controls (bbl.8 north east) and (btl.8 south east) .. (btl.center);
    
    \draw [cyan] (bbl.center) .. controls (bbl.8 north west) and (btl.8 south west) .. (btl.center);
    
    \end{scope}

    \begin{scope}[xshift=16cm]
    \node (boxtl) at (-3, 3) {};
    \node (boxtr) at (-1, 3) {};
    \node (boxbl) at (-3, -3) {};
    \node (boxbr) at (-1, -3) {};
    
    \draw [thick, green] (boxtl.center) -- (boxtr.center);
    \draw [thick, green] (boxtl.center) -- (boxbl.center);
    \draw [thick, green] (boxbl.center) -- (boxbr.center);
    \draw [thick, green] (boxbr.center) -- (boxtr.center);
    \end{scope}
    
    \begin{scope}[xshift=20cm]
    \node (boxtl) at (-3, 3) {};
    \node (boxtr) at (-1, 3) {};
    \node (boxbl) at (-3, -3) {};
    \node (boxbr) at (-1, -3) {};
    
    \draw [thick, green] (boxtl.center) -- (boxtr.center);
    \draw [thick, green] (boxtl.center) -- (boxbl.center);
    \draw [thick, green] (boxbl.center) -- (boxbr.center);
    \draw [thick, green] (boxbr.center) -- (boxtr.center);
    \end{scope}
    

    \begin{scope}[yshift=-1.5cm, xshift=24cm]
      
      \node (btl) at (-1, 1) {};
      \node (tbl) at (-1, 2) {};
      \node (btr) at (1, 1) {};
      \node (tbr) at (1, 2) {};
      
      \draw [thick, orange] (btl.center) .. controls (btl.4 north west) and (tbl.4 south west) .. (tbl.center) ;
      
       \draw [thick, cyan] (btr.center) .. controls (btr.4 north east) and (tbr.4 south east) .. (tbr.center) ;
      
    \end{scope}
    
    \begin{scope}[yshift=0cm, xshift=24cm]
    \node (ttl) at (-1, 2.5) {};
    \node (ttr) at (1, 2.5) {};
    \node (tbl) at (-1, 0.5) {};
    \node (tbr) at (1, 0.5) {};

    \node (btl) at (-1, -0.5) {};
    \node (btr) at (1, -0.5) {};
    \node (bbl) at (-1, -2.5) {};
    \node (bbr) at (1, -2.5) {};
    \node (leftcenter) at (-4.5, -0.75) {};
    \node (rightcenter) at (4.5, -0.75) {};

    \draw [orange] (bbl.center) .. controls (bbl.4 north west) and (leftcenter.4 north east) .. (ttl.center);
    
    \draw [cyan] (bbr.center) .. controls (bbr.4 north east) and (rightcenter.4 north west) .. (ttr.center);

    \draw [thick, cyan, -{Stealth}] (tbr.center) .. controls (tbr.8 north west) and (ttr.8 south west) .. (ttr.center) ;

    \draw [cyan, -{Stealth}] (bbr.center) .. controls (bbr.8 north west) and (btr.8 south west) .. (btr.center);
    
     \draw [orange, -{Stealth}] (bbl.center) .. controls (bbl.8 north east) and (btl.8 south east) .. (btl.center);

    \draw [orange, -{Stealth}] (tbl.center) .. controls (tbl.8 north east) and (ttl.8 south east) .. (ttl.center);

    \end{scope}

    \begin{scope}[xshift=24cm]
    \node (boxtl) at (-3, 3) {};
    \node (boxtr) at (-1, 3) {};
    \node (boxbl) at (-3, -3) {};
    \node (boxbr) at (-1, -3) {};
    
    \draw [thick, green] (boxtl.center) -- (boxtr.center);
    \draw [thick, green] (boxtl.center) -- (boxbl.center);
    \draw [thick, green] (boxbl.center) -- (boxbr.center);
    \draw [thick, green] (boxbr.center) -- (boxtr.center);
    \end{scope}
    
    \begin{scope}[xshift=28cm]
    \node (boxtl) at (-3, 3) {};
    \node (boxtr) at (-1, 3) {};
    \node (boxbl) at (-3, -3) {};
    \node (boxbr) at (-1, -3) {};
    
    \draw [thick, green] (boxtl.center) -- (boxtr.center);
    \draw [thick, green] (boxtl.center) -- (boxbl.center);
    \draw [thick, green] (boxbl.center) -- (boxbr.center);
    \draw [thick, green] (boxbr.center) -- (boxtr.center);
    \end{scope}
    
    \end{scope}
    
    \end{tikzpicture}

    \caption{\textbf{Top row:} Component possibilities for a Matching pair; \hspace{37pt}
    \textbf{Bottom row:} Components in the two new smaller diagrams that result from smoothing the Matching pair}
    \label{component possibilities}
\end{figure}

Observe that $D'$ and $D''$ are Balanced link diagrams themselves, and each has strictly fewer $A$-circles than our original $D$. So by the induction hypothesis, each matching pair in $D'$ (and in $D''$) is synchronized. Then by our previous Lemma \ref{synchronized means for same A:  same B iff same c}, we have that $D'$ and $D''$ are both synchronized.

By looking at the options of combinations in Figures \ref{B-possibilities} and \ref{component possibilities}, we see that then the diagrams $D'$ and $D''$ could only have come into existence by smoothing a synchronized matching pair in $D$. But since we chose an arbitrary matching pair in $D$ to smooth, this means that every matching pair in our larger diagram $D$ is a synchronized matching pair. 
\end{proof}

\vspace{7pt}

\begin{theorem}\label{arcs on same A-circle are part of same B-circle if and only if they are part of the same component}
 Balanced diagrams are synchronized.
\end{theorem}

\begin{proof}
Follows from Lemma \ref{Every matching pair is synchronized} and Lemma \ref{synchronized means for same A:  same B iff same c}
\end{proof}

\vspace{7pt}

Now at last we are ready for Theorem \ref{comp=Bcircles}:

\textbf{Theorem \ref{comp=Bcircles}}. \textit{Let $D$ be a Balanced link diagram. Then $n(D)=|B(D)|$. (The number of components is equal to the number of $B$-circles.)}

\begin{proof}[Proof of Theorem \ref{comp=Bcircles}]

Let $D$ be a Balanced diagram. Then the reduced $A$-state graph of $D$ is a tree, and therefore has a leaf. Since the set of matching pairs of $D$ is in bijection with the set of crossings in the reduced $A$-state graph, this means that $D$ contains a "leaf $A$-circle" - an $A$-circle that only touches a single matching pair. 

Therefore, just as a tree of $k+1$ vertices can be viewed as the result of attaching a single vertex to a tree of $k$ vertices, so too can a Balanced diagram $D$ where $|A(D)|=k+1$ be viewed as the result of attaching a single $A$-circle to a Balanced diagram $D'$ with $|A(D')|=k$. 

Now we proceed by induction on the number of $A$-circles. Observe that in the base case of $|A|=1$, the result holds. Let $D$ be a Balanced diagram with $|A(D)|=k+1$. Then there exists some Balanced diagram $D'$ such that $|A(D')|=k$, and an $A$-circle $A'$ in $D'$ such that $D$ can be obtained by grafting a leaf onto $D'$ along two (not necessarily distinct) arcs $x$ and $y$ of $A'$. 

By inductive hypothesis, the claim is true for $D'$ so $|B(D')|=n(D')$. 

Consider those arcs $x$ and $y$ in $D'$. If they are part of the same link component, then we create a new component by attaching a leaf circle. If instead $x$ and $y$ are part of different link components, then attaching a leaf combines these two components into one. Similarly, if $x$ and $y$ are part of the same $B$-circle, then we create a new $B$-circle by attaching a leaf, and if $x,y$ are part of different $B$-circles, then attaching a leaf combines them into one. 

On the surface, we seem to be looking at four different possibilities of changes in number of components and number of $B$-circles that could result from attaching a leaf circle.  However, we know from Theorem \ref{arcs on same A-circle are part of same B-circle if and only if they are part of the same component} that the diagram is synchronized, which means that arcs on the same $A$-circle are from the same $B$-circle if and only if they are from the same component. This leaves us with exactly two possibilities: that adding a leaf creates a new component and a new $B$-circle, or adding a leaf combines two components and combines two $B$-circles. 

By our inductive hypothesis, we know that $D'$ has the same number of $B$-circles as components. So when we add a leaf to obtain $D$ we either increase both by 1 or we decrease both by 1 - but in either case, the number of $B$-circles will still equal the number of components in our new diagram $D$.

So by induction, in any Balanced diagram the number of link components is equal to the number of $B$-circles.

\end{proof}

\end{document}